\documentclass[a4paper, 10pt, reqno]{amsart}
\usepackage[dvipsnames]{xcolor}
\usepackage[T1]{fontenc}
\usepackage{hyperref}
\hypersetup{
colorlinks = True,
linkcolor=Bittersweet,
citecolor=blue,
}
\usepackage{amssymb}
\usepackage{amsthm}
\usepackage{bm}
\usepackage{latexsym}
\usepackage{amsmath}

\usepackage{eufrak}
\usepackage{mathrsfs}
\usepackage{caption}
\usepackage{xparse}
\usepackage{amscd}
\usepackage[all,cmtip]{xy}

\usepackage{placeins}
\usepackage{afterpage}
\usepackage{amscd}
\usepackage{float}
\usepackage{graphics}
\usepackage{tikz}
\usetikzlibrary{graphs}
\usetikzlibrary{decorations.pathreplacing,angles,quotes}
\usetikzlibrary{shapes.geometric}

\usepackage{tikz-cd}
\usepackage{comment}
\usepackage{xspace}
\usepackage{mathtools}
\usepackage{pifont} %
\usepackage{amssymb}
\usepackage{amsthm}
\usepackage{graphicx}
\usepackage{subfig}
\usepackage{enumerate}
\usepackage{hyperref}

\usetikzlibrary{patterns,decorations.pathreplacing}
\usepackage{marginnote}

\usepackage{cancel}

\theoremstyle{plain}

\newtheorem{theorem}{Theorem}[section]
\newtheorem{proposition}[theorem]{Proposition}
\newtheorem{lemma}[theorem]{Lemma}
\newtheorem{corollary}[theorem]{Corollary}

\theoremstyle{definition}

\newcommand{\appsection}[1]{\let\oldthesection\thesection
\renewcommand{\thesection}{Appendix \oldthesection}
\section{#1}\let\thesection\oldthesection}

\newtheorem{definition}[theorem]{Definition}

\theoremstyle{remark}

\newtheorem{remark}[theorem]{Remark}
\newtheorem{example}[theorem]{Example}

\def\Q{{\mathbb{Q}}}
\def\C{{\mathbb{C}}}
\def\P{{\mathbb{P}}}

\def\O{{\mathcal{O}}}

\usepackage{xcolor}

\newcommand{\jr}[1]{\textsf{\scriptsize\color{blue} J: {#1}}}

\newcommand{\ff}[1]{\textsf{\scriptsize\color{orange} F: {#1}}}

\tikzset{every loop/.style={min distance=7mm,in=120,out=60,looseness=10}}

\begin{document}

\title{Optimal bounds for many T-singularities in stable surfaces}
\author[Fernando Figueroa]{Fernando Figueroa}
\email{fzamora@math.princeton.edu}
\address{Department of Mathematics, Fine Hall, Washington Road, Princeton NJ 08544-1000 USA.}

\author[Julie Rana]{Julie Rana}
\email{julie.f.rana@lawrence.edu}
\address{Department of Mathematics, Lawrence University, 711 E. Boldt Way, Appleton WI 54911, USA.}

\author[Giancarlo Urz\'ua]{Giancarlo Urz\'ua}
\email{urzua@mat.uc.cl}
\address{Facultad de Matem\'aticas, Pontificia Universidad Cat\'olica de Chile, Campus San Joaqu\'in, Avenida Vicu\~na Mackenna 4860, Santiago, Chile.}


\begin{abstract}
We effectively bound T-singularities on non-rational projective surfaces with an arbitrary amount of T-singularities and ample canonical class. This fully generalizes the previous work for the case of one singularity, and illustrates the vast increase in combinatorial complexity as the number of singularities grows. We find that certain combinatorial configurations lead to relatively high bounds. We classify all such configurations, and show that their non-existence gives a strong and optimal bound. As an application, we work out in detail the case of two singularities.   
\end{abstract}

\maketitle


\section{Introduction} \label{intro}

The KSBA moduli space $\overline{M}_{K^2,\chi}$, introduced by Koll\'ar and Shepherd-Barron in \cite{KSB88} and proved compact by Alexeev in~\cite{A94}, is the compactification of Gieseker's moduli space of surfaces of general type with fixed invariants $K^2$, $\chi$. \textit{T-singularities} are the complex 2-dimensional quotient singularities that admit a one-parameter $\Q$-Gorenstein smoothing, and thus form an important class of singularities that appear on surfaces parametrized by $\overline{M}_{K^2,\chi}$. In particular, every degeneration of a canonical surface of general type to a surface with only quotient singularities must be a surface with ample canonical class and only T-singularities. Such degenerations turn out to be abundant (see e.g. \cite{SU16,RU22,FPRR21}), and indeed give rise to large components of  $\overline{M}_{K^2,\chi}$. They have been used to solve difficult questions in complex geometry (see e.g. \cite{LP07,PPS09a,PPS09b}). Since $\overline{M}_{K^2,\chi}$ is compact, there is a finite list of T-singularities that may appear on the singular surfaces it parametrizes, including the ones which are not smoothable. \textit{What is this list of T-singularities after we fix $K^2$ and $\chi$?}



Effective results for this question are difficult to obtain. In \cite{RU17}, the second and third authors fully classified the case with one T-singularity. \textit{The purpose of the present paper is to work out the case of arbitrarily many T-singularities. We make no assumptions on smoothability of surfaces containing such singularities}. We will see that even in passing from 1 singularity to 2 singularities, the combinatorial complexity increases significantly. This is shown in Theorem \ref{optimal l=2}, where we work out the case of 2 singularities using the combinatorial machinery developed in the present paper. On the other hand, independent bounds for each singularity in $W$ are significantly weaker than bounds which consider all singularities at once, see Remark \ref{emphasys}.

T-singularities are divided into either ADE singularities, or cyclic quotient singularities $\frac{1}{dn^2}(1,dna-1)$ with $0<a<n$ coprime integers, and $d\geq 1$ \cite[Proposition 3.10]{KSB88}. When $d=1$ we call them \textit{Wahl singularities}, as they were originally considered by Wahl in his study of singularities admitting smoothings with Milnor number equal to zero \cite{W81}. We will bound T-singularities which are not ADE, i.e. whose $\Q$-Gorenstein smoothings do not admit simultaneous resolutions. A \textit{T-chain} is the exceptional divisor of the minimal resolution of a non-ADE T-singularity. Due to Wahl, we know that T-chains are formed via a particular algorithm, and they have special 
properties, some of them collected in Proposition~\ref{T-chain}. A non-ADE T-singularity $\frac{1}{dn^2}(1,dna-1)$ has index $n$, which is bounded by the $(r-d)$-th Fibonacci number, where $r$ is the number of curves of the associated T-chain (see \cite[Introduction]{RU17}). Let $W$ be a normal projective surface with $K_W$ ample and $l$ T-singularities $\frac{1}{d_i n_i^2}(1,d_i n_i a_i-1)$, whose T-chains have length $r_i$. Then bounding the sum of $r_i-d_i$ with respect to $K^2$ and $\chi$ will give a bound on the possible types of T-singularities allowed on $W$. 

We note two things at this point. First, $\chi=\chi(\O_W)$ is already bounded by $K_W^2$ via a generalized Noether's inequality \cite{TZ92}. Second, we can write down the following bound for the $d_i$ using the log Bogomolov-Miyaoka-Yau inequality: $$\sum_{i=1}^l d_i \leq 12 \chi(\O_W) - \frac{3}{4} K_W^2 + \sum_{i=1}^l \frac{1}{d_i n_i^2}.$$ This also bounds the number of T-singularities $l$. Therefore, the purpose of this paper is to bound the lengths $r_i$ of the T-chains with respect to $K^2$. 


Here is our general set-up. Let us consider the diagram
$$ \xymatrix{  & X  \ar[ld]_{\pi} \ar[rd]^{\phi} &  \\ S &  & W}$$
where the morphism $\phi$ is the minimal resolution of $W$, and $\pi$ is a composition of $m$ blow-ups such that $S$ has no $(-1)$-curves. Let $E_i$ be the pull-back divisor in $X$ of the $i$-th point blown-up through $\pi$. Thus $E_i$ is a connected, possibly non-reduced tree of $\P^1$'s, $E_i^2=-1$, and $E_i\cdot E_j=0$ for $i\neq j$. Let $C=\sum_{i=1}^l C_i$ be the exceptional (reduced) divisor of $\phi$, where $C_i=\sum_{j=1}^{r_i}C_{i,j}$ is the T-chain of the singularity $\frac{1}{d_in_i^2}(1,d_in_ia_i-1)$.

To obtain bounds, it is key to pay attention to the intersections $E_i \cdot C.$ Our first results state that $E_i \cdot (\sum_{C_{k,j} \nsubseteq E_i} C_{k,j})\ge 2$ (Lemma~\ref{S0} and Lemma~\ref{S1}).
This is the main goal of Section~\ref{s3}, and it already implies a strong constraint on the non-empty configurations of rational curves $\pi(C)$, namely (see Corollary~\ref{first}) that 
$$K_S \cdot \pi(C) \leq K_W^2 - K_S^2 + l,$$ with equality if and only if $X=S$. See Example~\ref{firstex} for a powerful application when $S$ is a non-rational elliptic fibration.

In Section~\ref{s3}, we prove that we always have $E_i \cdot C \geq 1$ for every $i$. If in addition $E_i \cdot C \geq 2$ for every $i$, then we can show (see Proposition~\ref{r-d}) that $$\sum_{j=1}^l (r_j - d_j) \leq 2(K_W^2-K_S^2) - K_S \cdot \pi(C).$$ As $K_W$ is ample, the surface $S$ cannot be a $\P^1$-fibration over a nonrational curve, and so $S$ is either rational or $K_S$ is nef. In the first case, we have the new problem of effective bounds for $K_S \cdot \pi(C)$; in the second case, we have obtained a strong and optimal bound. 

\begin{remark}
 By the results of Evans and Smith \cite{ES17}, we always have that $r_i \leq 4 K_W^2+1$ when all $d_i=1$ and $W$ is not rational. (This is essentially the bound in \cite{RU17} when we have one Wahl singularity.) But our work here will show that we can be much more effective when we consider interactions between the T-chains and the exceptional divisors of a minimal model of the minimal resolution of $W$. More precisely, when there are no $E_i$ with $E_i \cdot C=1$, then we have $\sum_{j=1}^l (r_j - d_j) \leq 2(K_W^2-K_S^2) - K_S \cdot \pi(C)$, as we show in Corollary~\ref{first}. The main work of Section~\ref{s4}, then, is to classify all such $E_i$s (see Theorem ~\ref{classificationEC=1}). In Section~\ref{s5}, we build on this classification and use Theorem \ref{maintheorem} to limit the presence of such exceptional curves. \label{emphasys}
\end{remark}

To attempt an effective bound more generally, we need to understand the $E_i$ with $E_i \cdot C=1$. In Section~\ref{s4} we completely classify them. Our strategy relies on various combinatorial arguments about the interactions between the dual graphs of the curves in $E_i$ and $C$. 

 \begin{theorem}[Theorem~\ref{classificationEC=1}] Every $E_i$ with $E_i \cdot C=1$ is one of the 9 types described from Example~\ref{T.2.1} to Example~\ref{C.2}.
\label{main on bad Ei}
\end{theorem}

As these are important building blocks, we explicitly describe these $E_i$ and their intersections with $C$ using the dual graphs shown in Tables~\ref{table1},~\ref{table2},~\ref{table3}.\footnote{We note that some related configurations appear in Kawamata's work \cite[Theorem 4.2]{K92} in the context of T-singularities in degenerations of elliptic fibrations \cite{K92}. They can be considered as the classification for $K_W$ nef with $K_W^2=0$. Our case is richer as $K_W^2>0$.} 

Finally, in Section~\ref{s5} we put all of this together in the form of a  new decorated graph, which is defined by all the maximal $E_i$ with $E_i \cdot C=1$. To find global bounds, we first need to find ``local'' bounds for each maximal $E_i$. In that process, we assume that $K_S$ is nef.

The global bound is obtained by considering each of the connected components of the decorated graph. We prove in Proposition~\ref{shape} that there are three possible shapes G1, G2, G3 for these components. We denote by $G'_i$ a connected component of the decorated graph containing $l'_i$ vertices, each of which represents a T-chain in $C$. We then define $l(G'_i)=2l'_i$ if $G'_i$ is of type G2 or G3 with a $(C.1)$; $l(G'_i)=l'_i$ otherwise. Then letting $R'_i-D'_i$ be the sum of the $r_j-d_j$ corresponding to the T-chains in $G'_i$, we have the following global bound. Note that each $f(G_i')$ is between $0$ and $3$; see Section~\ref{s5} for the precise definition, which depends on the shape of $G_i'$.

\begin{theorem} [Theorem~\ref{maintheorem}]
Let $W$ be a (not necessarily smoothable) surface with only T-singularities and $K_W$ ample. Assume that $W$ is not rational. Let $S$ be the minimal model of a resolution of $W$. Then $$\sum_i \frac{1}{l(G'_i)} (R'_i-D'_i) \leq 2(K_W^2-K_S^2)+ \sum_i f(G'_i) - K_S \cdot \pi(C).$$
\label{maintheorem intro}
\end{theorem}

To illustrate the significance of this result, we follow the proof of this theorem with two useful corollaries; see Section~\ref{s5}. 

It is not hard to show that asymptotically these inequalities are optimal (see Example~\ref{Optimal Case}), but an optimal bound for an arbitrary fixed $l$ would be cumbersome. This is because the interactions between the $E_i$ with $E_i \cdot C=1$ and $C$ can be arbitrarily complicated. These interactions have an immediate effect on our simple local inequalities, and so optimal bounds are lost in that process. Nevertheless, when one fixes the number of T-chains, one can use the work developed throughout this paper to obtain optimal bounds, although the difficulty of this process grows as we increase $l$. In Section~\ref{s6}, we finish this paper by describing explicitly the case $l=2$ in general (see Theorem~\ref{optimal l=2}), and realizing optimal examples.

\vspace{0.1cm}
The optimal boundedness problem for T-singularities in rational surfaces remains open. (We note that this is a particular case for the invariants $p_g=q=0$.)  Indeed, even in the case of one singularity, the known inequalities depend on the degree of the image of the exceptional divisor $C$ in $S$. See \cite{RU17} for more details.

\subsubsection*{Acknowledgments} We thank Javier Reyes for useful discussions. The third author thanks Lawrence University for their hospitality during his visit in April-May 2023. The second author was supported by NSF LEAPS-MPS grant 2137577. The third author was supported by the FONDECYT regular grant 1230065.

\tableofcontents


\section{Preliminaries} \label{prel}

A \textit{cyclic quotient singularity} $Y$, denoted by $\frac{1}{m}(1,q)$, is a germ at the origin of the quotient of $\mathbb{C}^2$ by $(x, y) \mapsto (\mu x, \mu^qy)$, where $\mu$ is a primitive $m$-th root of $1$, and $q$ is an integer with $0 < q < m$ and $gcd(q, m)=1$. Let $\sigma \colon \widetilde{Y} \rightarrow Y$ be the minimal resolution of $Y$, and let $C_1,\ldots, C_r$ be the exceptional curves of $\sigma$ (which form a chain of $\P^1$'s). The numbers $C_i^2=-b_i$ are computed using the Hirzebruch-Jung continued fraction
$$ \frac{m}{q}= b_1 - \frac{1}{b_2 - \frac{1}{\ddots - \frac{1}{b_r}}} =: [b_1, \ldots ,b_r].$$ We define its {\it length} as $r=r(q,m)$.

In the rest of this paper, the symbol $[b_1,\ldots,b_r]$ will correspond to the continued fraction, the singularity, or the chain of curves $C_1,\ldots, C_r$, depending on the context. Also $$[b_1,\ldots,b_r]-c-[b'_1,\ldots,b'_{r'}]$$ will represent a chain of $\P^1$'s with self-intersections $-b_i$, $-c$, and $-b'_i$, respectively.

As usual, one can write the numerical equivalence $$K_{\widetilde{Y}} \equiv \sigma^*(K_Y) + \sum_{i=1}^r \delta_i C_i$$ where $\delta_i \in ]-1,0]$ is (by definition) the \textit{discrepancy} at $C_i$. For cyclic quotient singularities, the discrepancies can be computed explicitly using adjunction.

A {\it T-singularity} is a quotient singularity that admits a $\Q$-Gorenstein one parameter smoothing \cite[Definition 3.7]{KSB88}. They are precisely either ADE singularities or $\frac{1}{dn^2}(1,dna-1)$ with $d \geq 1$, $0<a<n$ and gcd$(n,a)=1$ \cite[Proposition 3.10]{KSB88}. We call the exceptional divisor of a non-ADE T-singularity a {\it T-chain}.

A particular combinatorial structure governs T-singularities. For reference, we describe this and other well-known properties in the following proposition.

\begin{proposition}
For non-ADE T-singularities $\frac{1}{dn^2}(1,dna-1)$ we have:

\begin{itemize}
    \item[(i)] If $n=2$ then the T-chain is either $[4]$ or $[3,2, \ldots,2,3]$, where the number of $2$'s is $d-2$. In this case, all discrepancies are equal to $-\frac{1}{2}$.
    \item[(ii)] If $[b_1,b_2 \ldots, b_r]$ is a T-chain for a given $d$, then $[2, b_1, \ldots, b_{r-1},$ $b_r+1]$ and $[b_1+1,b_2, \ldots,b_r,2]$ are T-chains for the same $d$.
    \item[(iii)] Every non-ADE T-chain can be obtained by starting with one of the singularities in (i) and iterating the steps described in (ii).
    \item[(iv)] Consider a T-chain $[b_1, \ldots, b_r]=\frac{dn^2}{dna-1}$ with discrepancies $$-1+\frac{t_1}{n}, \ldots, -1+\frac{t_r}{n}$$ respectively, where $t_1+t_r=n$. Then $[b_1+1,b_2, \ldots,b_r,2]$ has discrepancies $-1+\frac{t_1}{n+t_1},\ldots,-1+\frac{t_r}{n+t_1},-1+\frac{t_1+t_r}{n+t_1}$, and $[2,b_1, \ldots,b_r+1]$ has discrepancies $-1+\frac{t_1+t_r}{n+t_r},$ $-1+\frac{t_1}{n+t_r},\ldots,-1+\frac{t_r}{n+t_r}$, respectively.
    \item[(v)] Given the T-chain $[b_1,\ldots,b_r]$, the discrepancy of an ending $(-2)$-curve is $>-\frac{1}{2}$.
\end{itemize}

\label{T-chain}
\end{proposition}

\begin{proof}
The points (i), (ii) and (iii) are \cite[Proposition 3.11]{KSB88}. The point (iv) is \cite[Lemma 3.4]{S89}. The point (v) is a simple consequence of (iv).
\end{proof}

\begin{definition}
For a non-ADE T-singularity $\frac{1}{dn^2}(1,dna-1)$, we define its \textit{center} to be the collection of exceptional divisors in the corresponding T-chain which have the lowest discrepancy, that is the curves in the center all have discrepancy $-\frac{n-1}{n}$.
\label{centro}
\end{definition}

Hence, the divisors in the center are the ones corresponding to (i) after several applications of algorithm (ii). An important characteristic of the center for us is the following.

\begin{proposition}
Let $[b_1,\ldots,b_r]-1-[b'_1,\ldots,b'_{r'}]$ be a chain of $\P^1$'s where $[b_1,\ldots,b_r]$ and $[b'_1,\ldots,b'_{r'}]$ are T-chains, and $\delta_r+\delta'_1 \leq -1$ \footnote{In particular, when $\delta_r+\delta'_1 < -1$ this is the situation of a P-resolution \cite{KSB88}.}. If we contract the $(-1)$-curve and all new $(-1)$-curves after that, then no curve in the center of either T-chain is contracted.
\label{star}
\end{proposition}

\begin{proof}
We suppose for a contradiction that at least one curve in a center is contracted. Without loss of generality, suppose the center of  $[b_1, \ldots, b_r]$ is the  first to have a contracted curve. Let us consider as notation: $$[b_1, \ldots, b_s, \text{center}, b_t, \ldots, b_r]-1- [b_1', \ldots, b_{r'}'].$$

In order for $b_{t-1}$ to be contracted, the curves $[b_t, \ldots, b_r]$ must be blown-down, and $b_{t-1}$ must eventually become $1$.
As the blow-down process acts like the inverse of the algorithm (ii) in  Proposition~\ref{T-chain}, over the curves $[b_l, \ldots,b_r]$ and $[b_1', \ldots, b_{l'}']$. We must have that both T-chains are obtained from the same $s$ iterations of algorithm (ii) in Proposition~\ref{T-chain}: the first starting from $[x_1,\ldots,x_d]$, a T-chain with $n=2$ (from (i)), and the other starting from some chain $[\beta_1, \ldots \beta_k]$. Then an intermediate blow-down is: $$[b_1, \ldots, b_{t-2},x_d+y]-1-[\beta_1,b_{s+2}' \ldots, b_{r'}']$$ where $y \geq 0$ and $x_d=3$ or $4$. Then $\beta_1=2$, because we must eventually contract $x_d+y$. In particular, $\delta(\beta_1)>-\frac{1}{2}$ by part (v) Proposition~\ref{T-chain}.

Observe that when we construct $[b_1',\ldots,b_{r'}']$ and $[b_1,\ldots,b_{r}]$ from $[\beta_1, \ldots \beta_k]$ and $[x_1,\ldots,x_d]$ using the algorithm of Proposition~\ref{T-chain}, we can use part (iv) of Proposition~\ref{T-chain} to track the discrepancy of the leftmost divisor in each chain. That is, if $[e_1,\ldots,e_s]$ is a T-chain with discrepancy $\delta$ at $e_1$, then $[e_1+1,\ldots,e_s,2]$ has discrepancy $\frac{-1}{\delta+2}$ at $e_1+1$ and $[2,e_1,\ldots,e_s+1]$ has discrepancy $\frac{\delta}{1-\delta}$ at $2$. The key for us is that both of these are increasing functions of $\delta$ with values in the interval $[-1,0]$.

In particular, because we have $\delta_1'(\beta_1) > -\frac{1}{2}=\delta_1(x_1)$ and we apply the same algorithm to both chains, to obtain $[b_1,\ldots,b_r]$ and $[b'_1,\ldots,b'_{r'}]$.  It must also be the case that $\delta_1'>\delta_1$. But adding $\delta_r$ to both sides of this inequality, we find that $$\delta_1'+\delta_r> \delta_1+\delta_r=-1,$$ contradicting our hypothesis that $\delta_r+\delta'_1 \leq-1$.
\end{proof}


\begin{corollary}

Let us consider the chain of $\mathbb{P}^1$'s
$$[b_{1,1},\ldots,b_{1,r_1}]-1-[b_{2,1},\ldots,b_{2,r_2}]-1-\ldots-1-[b_{x,1},\ldots,b_{x,r_x}]$$
where the $[b_{i,1},\ldots,b_{i,r_i}]$ are T-chains, and $\delta_{j,r_j}+\delta_{j+1,1} \leq -1$ for all $j=1,\ldots,x-1$. If we contract the $(-1)$-curves and all new $(-1)$-curves after that, then no curve in the center of any T-chain is contracted.


\label{starc}
\end{corollary}

\begin{proof} If the center of each T-chain has at least two curves, then we can apply Proposition~\ref{star} separately on the left and right of each T-chain and obtain the result. So it is enough to prove the statement in the case that the center of at least one T-chain consists of only one curve. Thus we let $j$ be one integer for which the T-chain $[b_{j,1},\ldots,b_{j,r_j}]$ has center consisting of only one curve, that is, the center of this $j$th T-chain consists of a single curve of self-intersection at most $-4$.

Suppose that $[b_{j,1},\ldots,b_{j,r_j}]$ is obtained from the T-chain $[4]$ by a sequence of steps, following the algorithm in Proposition~\ref{T-chain}, and consider the related T-chain $[b'_{j,1},\ldots,b'_{j,r_j+1}]$ obtained from the T-chain $[3,3]$ by following the exact same sequence of steps used to obtain the T-chain $[b_{j,1},\ldots,b_{j,r_j}]$ from $[4]$. 

Observe that, apart from the curves in the center, both T-chains are the same, i.e. the curves outside of the center appear in the same order on the left and right of the center, with the same self-intersections and the same discrepancies. The only difference between the two T-chains is that the first has center $[4+k]$, and the other has center $[3,3+k]$ or $[3+k, 3]$, for some $k\ge 0$. Without loss of generality, we can assume that the center is $[3,3+k]$.

Replace the $j$th T-chain in the original chain of $\mathbb{P}^1$s with the chain $[b'_{j,1},\ldots,b'_{j,r_j+1}]$. Then applying Proposition~\ref{star} separately to the left and right of this chain, we see that the center of this chain is not contracted. In particular, it becomes at worst the chain $[2,2]$. Thus, this center eventually intersects at most $k+2$ $(-1)$-curves; that is, at most one on the left and at most $k+1$ on the right.

Thus, we see that the center $[4+k]$ of the original chain $[b_{j,1},\ldots,b_{j,r_j}]$ can be affected by at most $k+2$ contractions, so it has self-intersection at most $-2$ after all the blow-downs. 
\end{proof}

An easy observation is that, after all the blow-downs in the chain of Corollary~\ref{starc}, the center of the leftmost and the rightmost T-chains will each end up with at least one curve that has self-intersection strictly less than $-2$. This will be useful in some proofs where we have a $(-1)$-curve attached to some end of this chain.

\section{General set-up and first constraints} \label{s3}

Let $W$ be a normal projective surface with $K_W$ ample and only T-singularities $\frac{1}{d_in_i^2}(1,d_in_ia_i-1)$ where $i \in \{1,\ldots,l\}$. 

Let us consider the diagram
$$ \xymatrix{  & X  \ar[ld]_{\pi} \ar[rd]^{\phi} &  \\ S &  & W}$$
where the morphism $\phi$ is the minimal resolution of $W$, and $\pi$ is a composition of $m$ blow-ups such that $S$ has no $(-1)$-curves. We use the same notation as in \cite{R14,RU17}. Let $E_i$ be the pull-back divisor in $X$ of the $i$-th point blown-up through $\pi$. Therefore, $E_i$ is a connected, possibly non-reduced tree of $\P^1$'s, $E_i^2=-1$, and $E_i\cdot E_j=0$ for $i\neq j$. 

Let $$C=\sum_{i=1}^lC_i=\sum_{i=1}^l \sum_{j=1}^{r_i}C_{i,j}$$ be the exceptional (reduced) divisor of $\phi$, where $C_i=\sum_{j=1}^{r_i}C_{i,j}$ is the T-chain of the singularity $\frac{1}{d_in_i^2}(1,d_in_ia_i-1)$. Using the particular recursive properties of T-chains, one can show the well-known formula
\begin{equation}
    K_S^2-m+\sum_{i=1}^l(r_i-d_i+1)=K_W^2.
    \label{basic equation}
\end{equation}

\begin{remark}
Throughout this work, we will assume that $m>0$, since otherwise $K_W^2-K_S^2=\sum_{i=1}^l(r_i-d_i+1)$, and this case will verify our main results.
\end{remark}

\begin{lemma}
For any $(-1)$-curve $\Gamma$ in $X$ we have $\Gamma \cdot C \geq 2$. For any $(-2)$-curve $\Gamma$ in $X$ not in $C$ we have $\Gamma \cdot C \geq 1$.\label{lem:ample-KW}
\label{(-1)-(-2)-curve}
\label{disc}
\end{lemma}

\begin{proof}
It is a simple computation using the pull-back of the canonical class, the discrepancies of the $C_{i,j}$, and that $K_W$ is ample. More precisely, $0 < \phi(\Gamma) \cdot K_W= \Gamma \cdot K_X - \sum \delta_{i,j} (C_{i,j} \cdot \Gamma)$, where the $-1<\delta_{i,j}<0$
are the discrepancies of $C_{i,j}$.
\end{proof}

\begin{lemma}
We have $\Big(\sum_{i=1}^m E_i \Big) \cdot C = \sum_{j=1}^l (r_j - d_j + 2) - K_S \cdot \pi(C)$.
\label{int}
\end{lemma}

\begin{proof}
Same as in \cite[Lemma 2.4]{RU17}.
\end{proof}

\begin{lemma}
For each $i$, we have $ E_i \cdot C \geq -1+E_i \cdot \Big(\sum_{C_{k,j} \nsubseteq E_i} C_{k,j} \Big).$
\label{weekbound}
\end{lemma}

\begin{proof}
If $C_{k,j} \subset E_i$, then $C_{k,j} \cdot E_i = 0$ or $C_{k,j} \cdot E_i = -1$. The latter case can happen for at most one $C_{k,j}$ in $C$.
\end{proof}

\begin{definition}
Let $S_h$ be the set of $E_i$s such that $E_i \cdot \Big(\sum_{C_{k,j} \nsubseteq E_i} C_{k,j} \Big)=h.$ Let $s_h$ be the cardinality of $S_h$.
\end{definition}

We note that the sets $S_h$ partition the set of $E_i$s, and so $m=\sum_{h\geq 0} s_h$.

\begin{corollary}
We have $\Big(\sum_{i=1}^m E_i \Big) \cdot C \geq -m+\sum_{h\geq 0} h s_h $.
\label{ineq}
\end{corollary}

\begin{proof}
This is adding up Lemma~\ref{weekbound} for each $E_i$.
\end{proof}

The inequality in the previous corollary is important in the following sense. Say that $s_0=s_1=\ldots=s_a=0$ for some $a \geq 1$. Then the inequality becomes
$$\Big(\sum_{i=1}^m E_i \Big) \cdot C \geq -m+\sum_{h\geq 0} h s_h \geq -m + (a+1) \sum_{h\geq a} s_h = am,$$
but we know that $m=-K_W^2+K_S^2 +\sum_{i=1}^l(r_i-d_i+1)$ and
$$\Big(\sum_{i=1}^m E_i \Big) \cdot C = \sum_{j=1}^l (r_j - d_j + 2) - K_S \cdot \pi(C),$$
and so by replacing in the inequality we obtain
\begin{equation}
    (a-1)\Big(\sum_{j=1}^l (r_j - d_j)\Big) \leq a(K_W^2-K_S^2) + (2-a)l - K_S \cdot \pi(C).
    \label{cool equation}
\end{equation}

\begin{corollary}
If $s_0=s_1=\ldots=s_a=0$ for some $a \geq 2$, then
$$\sum_{j=1}^l (r_j - d_j) \leq \frac{a}{a-1}(K_W^2-K_S^2) - \frac{a-2}{a-1} l - \frac{1}{a-1}K_S \cdot \pi(C).$$ \label{cool}
\end{corollary}

We will soon prove that $s_0=0$ and $s_1=0$, but unfortunately $s_2$ could be different than zero. We will spend a big part of this paper classifying the $E_i$s in $S_2$, to finally show an inequality that takes this case into account. We now define a key graph depending on $E_i$ and $C$, which will be used to prove various results.


\begin{definition}
Let us denote by $G_{E_i}$ the graph whose vertices are the curves in the support of the divisor $E_i+\sum C_{k,j}$, and with vertices attached by $t$ edges if the corresponding curves intersect with multiplicity $t$. We replace all vertices corresponding to curves in $E_i \cap C$ and $E_i \backslash C$ with $\Box$ and $\circ$ vertices, respectively, the rest are black dots. We omit from $G_{E_i}$ the T-chains which are disjoint from $E_i$, 
so that $G_{E_i}$ is connected. 
\label{GEi}
\end{definition}

We note that $G_{E_i}$ does not recover the divisor $E_i+\sum C_{k,j}$ with their intersections and multiplicities. When there is no ambiguity, we may occasionally use the same notation for curves in the support of a divisor and the corresponding vertices in its dual graph.

Recall that a graph is \emph{simple} if any pair of vertices is connected by at most one edge, and there are no loops (edges connecting vertices to themselves). We will see later than the cases of interest to us will be when the graph $G_{E_i}$ is indeed simple (see Proposition~\ref{Ei}).

\begin{remark}
A useful fact is that a $(-1)$-curve in $E_i$ cannot intersect three different curves in $E_i$, since blowing down this $(-1)$-curve would yield a triple point, but since $E_i$ is obtained by blowing up a smooth point on $S$, each blow up gives only nodes. In other words, a $(-1)$-curve cannot intersect three boxes in the graph $G_{E_i}$.
\label{-1 curve}
\end{remark}

\begin{lemma}
Suppose that $G_{E_i}$ is a tree and that a vertex $A$ in $E_i$ has more than two neighbors. If a component $H$ of $G_{E_i} \setminus A$ is contained in $E_i$, then $H$ cannot be contracted completely before $A$ is contracted. 
\label{3Neighbors}
\end{lemma}

\begin{proof}
Suppose that such $A$ exists, and let $H$ be the connected component of $G_{E_i}$ which contracts completely before $A$. 

We may assume that $H$ is a path graph in the following way. Suppose it is not and choose a node $B$ in $H$ with more than two neighbors that is farthest from $A$. By hypothesis, every neighbor of $B$ is in $E_i$, so in order to avoid a triple point while blowing down, one of the components of $E_i \setminus B$ must contract before $B$. This component $H'$ cannot be the one containing $A$, since $A$ contracts after $H$. Therefore, $H'$ is contained in $H$, so it is completely contained in $E_i$. Since $B$ was chosen farthest away from $A$, $H'$ cannot contain curves with more than two neighbors, so it is a path graph.

Since the component $H$ contracts independently of the other components of $E_i$, we may also assume that $H$ contracts before any curve in the other components of $E_i \setminus A$.

Since $H$ can be contracted, it must contain a $(-1)$-curve, and this curve must intersect two different T-chains, so $H$ fully contains at least one T-chain.

If $H$ contains at least two T-chains, this satisfies the hypothesis of Corollary~\ref{starc}, which leads to a contradiction.

If $H$ only contains a single T-chain, then thanks to the $(-1)$-curve in $H$ that should intersect two T-chains, $A$ needs to be part of another T-chain. Since $H$ gets contracted before any other curve, the union of this chain and $H$ satisfies the hypothesis of Corollary~\ref{starc}, which also leads to a contradiction.
\end{proof}

\begin{lemma}
We always have $s_0=0$.
\label{S0}
\end{lemma}

\begin{proof}

In order to have $s_0>0$ we must have some $E_i$ such that each T-chain either has no curve in $E_i$ and does not intersect it, or it is contained completely in $E_i$. The presence of a $(-1)$-curve in $E_i$ and Lemma~\ref{lem:ample-KW} implies the existence of a T-chain contained completely in $E_i$. 
For this $i$, the vertices of the graph $G_{E_i}$ correspond exactly to the curves in $E_i$. In particular, $G_{E_i}$ is a tree.

If $G_{E_i}$ is a path graph, then
by Corollary~\ref{starc} no divisor in a center can be contracted, a contradiction.

If $G_{E_i}$ is not a path graph, then there is a vertex $A$ with more than two neighbours. $A$ cannot be contracted before any of its neighbours, as that would create a triple point in a blow-down of $E_i$ and all blow-downs have only nodes. Therefore some component of $G_{E_i} \setminus A$ contracts before $A$. This contradicts Lemma~\ref{3Neighbors}.

\end{proof}

\begin{lemma}
We always have $s_1=0$.

\label{S1}
\end{lemma}

\begin{proof}
Let us assume that $s_1 >0$. Then there exists an $E_i$ such that $$E_i \cdot \Big(\sum_{C_{k,j} \nsubseteq E_i} C_{k,j} \Big) =1.$$ By definition, all the T-chains with no curve in $E_i$ and no curve intersecting $E_i$ get omitted from $G_{E_i}$, so $G_{E_i}$ is a connected simple graph.

We also claim that $G_{E_i}$ is a tree. This is because any potential cycle would contain vertices in $E_i$ and vertices not in $E_i$, which would give at least two points of intersection between curves in $E_i$ and curves in T-chains but not in $E_i$, but this is not possible by our assumption on $E_i\cdot \Big(\sum_{C_{k,j} \nsubseteq E_i} C_{k,j} \Big)$. 

Since $G_{E_i}$ is a tree, we can also say that there is more than one T-chain, because otherwise a $(-1)$-curve in $E_i$ would create a cycle in $G_{E_i}$. Using that $E_i \cdot \Big(\sum_{C_{k,j} \nsubseteq E_i} C_{k,j} \Big) =1$ we see that there is at least one T-chain completely contained in $E_i$. 

We now deal with two cases:

\textbf{(Case I):} Every vertex in $E_i$ is connected to at most two vertices in $G_{E_i}$. In particular, this means that $E_i$ is a path graph. Since $E_i$ contains both a $(-1)$-curve and a T-chain, we have a chain satisfying the hypotheses of Proposition~\ref{star}, 
so no center divisor can be contracted. This contradicts the fact that there is a T-chain contained in $E_i$. 



\textbf{(Case II):} Suppose there is a vertex $A$ in $E_i$ that is connected to at least three vertices in $G_{E_i}$. We consider the connected components of $G_{E_i}\setminus A$. 

Since $E_i \cdot \Big(\sum_{C_{k,j} \nsubseteq E_i} C_{k,j} \Big) =1$, we know that at least two of these components are contained in $E_i$. If any of these components contracts completely before $A$ does, then we can apply Lemma~\ref{3Neighbors} to arrive at a contradiction.
Hence none of these components contracts completely before $A$ does.  Thus, contracting $(-1)$-curves until $A$ becomes a $(-1)$-curve $\tilde{A}$, we have that $\tilde{A}$ has at least three neighbours.

By Remark~\ref{-1 curve}, at least one neighbour $B$ of $\tilde{A}$ must be in the image of $G_{E_i} \setminus E_i$.
On the other hand, since $E_i \cdot \Big(\sum_{C_{k,j} \nsubseteq E_i} C_{k,j} \Big) =1$, we see that $B$ is the only neighbour of $\tilde{A}$ that is in the image of $\Gamma_{E_{i}} \setminus E_i$. Since at least two neighbours of $\tilde{A}$ of are in the image of $E_i$ and are not contracted before $\tilde{A}$. We have that $A$ appears with coefficient at least $2$ in the divisor $E_i$. But, by pulling-back $B$ and $\tilde{A}$, this would contradict our assumption that $E_i \cdot \Big(\sum_{C_{k,j} \nsubseteq E_i} C_{k,j} \Big) =1$.

So, in each case we get a contradiction. Therefore $s_1=0$.
\end{proof}

\begin{corollary}
We always have $$K_S \cdot \pi(C) \leq K_W^2 - K_S^2 + l,$$ and we have equality if and only if $X=S$.
\label{first}
\end{corollary}

\begin{proof}
By Lemmas~\ref{S0} and~\ref{S1}, we know $s_0=s_1=0$, so we can set $a=1$ in Equation~\eqref{cool equation}, giving us the desired inequality. 

Now if $X=S$, then $K_S\cdot \pi(C)=K_S \cdot C= \sum_{j=1}^{l} (r_j-d_j+2)$, and so by adjunction and Equation~\eqref{basic equation}, we obtain $K_S \cdot \pi(C) = K_W^2 - K_S^2 + l$. 

On the other hand, if $K_S \cdot \pi(C) = K_W^2 - K_S^2 + l$, then we have $m-l=\Big(\sum_{i=1}^m E_i \Big) \cdot C$ by Lemma~\ref{int}. But we know that $\Big(\sum_{i=1}^m E_i \Big) \cdot C \geq m$ as $s_0=0$. Therefore $l=0$.
\end{proof}

\begin{example} 
Motivated by applications to moduli of surfaces, let us suppose that $S$ is a non-rational elliptic fibration. 
Since $K_W$ is ample, this elliptic fibration is over $\P^1$, and  $q(S)=0$. Therefore $K_S \equiv (p_g(S)-1+\sum_{i=1}^s \frac{m_i-1}{m_i}) F$, where $m_1,\ldots,m_s$ are the multiplicities of the $s$ multiple fibers in the elliptic fibration, and $F$ is the class of a fiber. (If there are no multiple fibers, then $K_S \equiv (p_g(S)-1)F$.) Let $\pi(C)=\sum_{i=1}^{r'} C'_i$ be the decomposition of $\pi(C)$ into irreducible curves. Let $a_j$ be the number of $C'_i$ with $C'_i \cdot F=j$. Then the inequality above implies $$\Big(\sum_{j\geq 1} j a_j \Big) \Big(p_g(S)-1+\sum_{i=1}^s \frac{m_i-1}{m_i} \Big) \leq K_W^2+l.$$ This can be used as a non-trivial condition on the curves $C'_i$ to construct a surface $W$ with $K_W$ ample, and $(K_W^2,l)$ fixed. Moreover, we have better control over this inequality if, a priori, we understand all the possible $E_i$s with $E_i \cdot C=1$.
\label{firstex}
\end{example}

\section{Classification of the $E_i$ with $E_i \cdot C=1$} \label{s4}

From now on, we will use the notation $$D:=\sum_{j=1}^{l} d_j, \ \ R:=\sum_{j=1}^{l} r_j, \ \ \text{and} \ \  \lambda:= K_S \cdot \pi(C).$$

\begin{proposition}
We have $R-D \leq 2(K_W^2-K_S^2) + Z-\lambda,$ where $Z$ is the number of $E_i$ such that $E_i \cdot C= 1$.
\label{r-d}
\end{proposition}

\begin{proof}
Recalling the proof of Lemma~\ref{weekbound}, if $E_i \in S_h$, then $E_i \cdot C$ is equal to $h-1$ or $h$ depending on the existence of $C_{j,k} \subset E_i$ or not, respectively. Since $s_0=s_1=0$ (Lemmas~\ref{S0} and~\ref{S1}), we have 
$$ \sum_{i=1}^m E_i \cdot C = \sum_{h \geq 0} \left( \sum_{E_i \in S_h} E_i \right) \cdot C = 2(s_2-Z)+Z + \sum_{h \geq 3} \left( \sum_{E_i \in S_h} E_i \right) \cdot C,$$ 
and so 
$$ \sum_{i=1}^m E_i \cdot C \geq 2 \sum_{h \geq 2} s_h -Z = 2m-Z,$$
since $m = \sum_{h \geq 0} s_h=\sum_{h \geq 2} s_h$. By Lemma~\ref{int}, we can write 
$$R-D +2l= \sum_{i=1}^m E_i \cdot C + \lambda \geq 2m -Z + \lambda,$$ 
and the result follows by replacing $m$ with $-(K_W^2-K_S^2)+R-D+l$.
\end{proof}

\begin{proposition}
Assume that $E_i \cdot C=1$. Then, $$E_i \cdot \Big(\sum_{C_{k,j} \nsubseteq E_i} C_{k,j} \Big) =2 \ \ \ \text{and} \ \ \ E_i \cdot \Big(\sum_{C_{k,j} \subseteq E_i} C_{k,j} \Big) =-1,$$ and $G_{E_i}$ is simple.
\label{Ei}
\end{proposition}

\begin{proof}
By Lemma~\ref{weekbound}, the equality $E_i \cdot C=1$ implies that either $$E_i \cdot \Big(\sum_{C_{k,j} \nsubseteq E_i} C_{k,j} \Big) =1 \ \ \ \text{and} \ \ \ E_i \cdot \Big(\sum_{C_{k,j} \subseteq E_i} C_{k,j} \Big) =0,$$ or $$E_i \cdot \Big(\sum_{C_{k,j} \nsubseteq E_i} C_{k,j} \Big) =2 \ \ \ \text{and} \ \ \ E_i \cdot \Big(\sum_{C_{k,j} \subseteq E_i} C_{k,j} \Big) =-1,$$ and the first case is impossible by Lemma~\ref{S1}.

Assume that $G_{E_i}$ is not simple. Since the supports of $E_i$ and $C$ are simple normal crossing divisors, $G_{E_i}$ not being simple implies the existence of a curve $A$ in $E_i$ not in $C$ which intersects a curve $B$ in $C$ but not in $E_i$ with multiplicity bigger than or equal to $2$. As $E_i \cdot (\sum_{C_{k,j} \nsubseteq E_i} C_{k,j}) =2$, we must have that $A$ has multiplicity $1$ in $E_i$ and $A \cdot (\sum_{C_{k,j} \nsubseteq E_i} C_{k,j})=2$. 

We also must have that each T-chain is either completely contained inside of $E_i$, or it does not contain any curve in $E_i$, since otherwise we would have an extra intersection with $\sum_{C_{k,j} \nsubseteq E_i} C_{k,j}$. In particular, since $E_i \cdot \big(\sum_{C_{k,j} \subseteq E_i} C_{k,j} \big) =-1$, we can say that $E_i$ contains at least one T-chain. Hence the component of $G_{E_i}\setminus{B}$ containing $A$ is the tree of curves $E_i$.

For the rest of the argument we only consider $E_i$. We have two cases. Suppose that $E_i$ is a path graph. Then we obtain a contradiction using Corollary~\ref{starc} (even in the case that $A$ is a $(-1)$-curve). Suppose now that $E_i$ has a vertex $V$ connected to three or more vertices in $E_i$. As no connected component in $E_i \setminus V$ can be contracted before $V$ does (see Lemma~\ref{3Neighbors}), we have that $A$ must have multiplicity at least $2$ in $E_i$, a contradiction. Therefore no such $A$ can exist.

\end{proof}

For the rest of this section, let $E_i$ be as in Proposition~\ref{Ei}. That is, we assume that $E_i \cdot C=1$, and so $$E_i \cdot \Big(\sum_{C_{k,j} \nsubseteq E_i} C_{k,j} \Big) =2 \ \ \ \text{and} \ \ \ E_i \cdot \Big(\sum_{C_{k,j} \subseteq E_i} C_{k,j} \Big) =-1,$$ and $G_{E_i}$ is simple. The goal is to classify all such $E_i$, finding all the combinatorial possibilities for $G_{E_i}$. The goal of the next section will be to bound $Z$. 

We now give a series of examples of such $E_i$, which will be the classification of all of them at the end. The node $\bigcirc$ in the next figures denotes a $(-1)$-curve in $E_i$.

\begin{example}
\textbf{(T.2.1)} In this case, $G_{E_i}$ is a tree with two T-chains and one $(-1)$-curve connecting them as shown in Figure~\ref{fT.2.1}. The divisor $E_i$ consists of $m>0$ $(-2)$-curves inside of one of the T-chains, and the $(-1)$-curve, which also intersects the other T-chain. 

\begin{figure}[htbp]
\scalebox{0.8}{
\centering
    \begin{tikzpicture}[square/.style={regular polygon,regular polygon sides=4},scale=1]

        \node at (1,-2) [circle, draw, fill=black] (a00) {};
        \node at (2,-2) [draw=none] (a0) {\ldots};
        \node at (3,-2) [circle, draw, fill=black] (a1) {};
        \node at (4,-2) [draw=none] (a2) {\ldots};
        \node at (5,-2) [circle, draw, fill=black] (a3) {};
        \node at (3,-1) [circle, draw] (a4) {};
        \node at (3,0) [square, draw, label=left:-2] (a5) {};
        \node at (4,0) [draw=none] (a6) {\ldots};
        \node at (5,0) [square, draw, label=below:-2] (a7) {};
        \node at (6,0) [circle, draw, fill=black] (a8) {};
        \node at (7,0) [draw=none] (a9) {\ldots};
        \node at (8,0) [circle, draw, fill=black] (a10) {};

        \draw (a00) -- (a0) -- (a1) -- (a2) -- (a3);
        \draw (a5) -- (a6) -- (a7) -- (a8) -- (a9) -- (a10);
        \draw (a1) -- (a4) -- (a5);

        \draw[decoration={brace,raise=10pt},decorate]
        (3,0) -- node[above=10pt] {$m$} (5,0);

    \end{tikzpicture}
}
    \caption{The T.2.1 graph} \label{fT.2.1}
\end{figure}
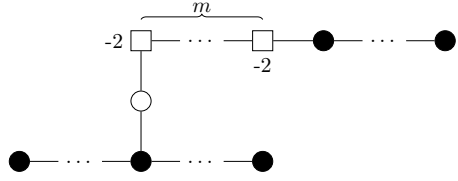
\label{T.2.1}
\end{example}

\begin{example}
\textbf{(T.2.2)} In this case, $G_{E_i}$ is a tree with two T-chains and one $(-1)$-curve connecting them as shown in Figure~\ref{fT.2.2}. The $\square$ are arbitrary, and there are $m \geq0$ for one T-chain, and $n \geq 0$ for the other. At least one of $m, n$ is nonzero. (If one of $m,n$ is zero, then we are in the case (T.2.1), so we can assume that both $m,n$ are non-zero.) 

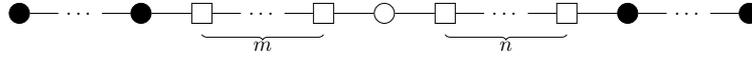
\begin{figure}[htbp]
\scalebox{0.8}{
\centering
    \begin{tikzpicture}[square/.style={regular polygon,regular polygon sides=4},scale=1]

        \node at (0,0) [circle, draw, fill=black] (a1) {};
        \node at (1,0) [draw=none] (a2) {\ldots};
        \node at (2,0) [circle, draw, fill=black] (a3) {};
        \node at (3,0) [square, draw] (a4) {};
        \node at (4,0) [draw=none] (a5) {\ldots};
        \node at (5,0) [square, draw] (a6) {};
        \node at (6,0) [circle, draw] (a7) {};
        \node at (7,0) [square, draw] (a8) {};
        \node at (8,0) [draw=none] (a9) {\ldots};
        \node at (9,0) [square, draw] (a10) {};
        \node at (10,0) [circle, draw, fill=black] (a11) {};
        \node at (11,0) [draw=none] (a12) {\ldots};
        \node at (12,0) [circle, draw, fill=black] (a13) {};

        \draw (a1) -- (a2) -- (a3) -- (a4) -- (a5) -- (a6) -- (a7) -- (a8) -- (a9) -- (a10) -- (a11) -- (a12) -- (a13);

        \draw[decoration={brace, mirror,raise=10pt},decorate]
        (3,0) -- node[below=10pt] {$m$} (5,0);
        \draw[decoration={brace, mirror,raise=10pt},decorate]
        (7,0) -- node[below=10pt] {$n$} (9,0);

    \end{tikzpicture}
}
    \caption{The T.2.2 graph} \label{fT.2.2}
\end{figure}

\label{T.2.2}
\end{example}


\begin{example}
\textbf{(T.2.3)} In this case, $G_{E_i}$ is a tree with two T-chains and one $(-1)$-curve connecting them as shown in Figure~\ref{fT.2.3}. The self-intersections of the curves in $E_i$ are shown in the figure, as well as the parameters $n >0$, $m \geq 0$. 

\begin{figure}[htbp]
\scalebox{0.8}{
\centering
    \begin{tikzpicture}[square/.style={regular polygon,regular polygon sides=4},scale=1]

        \node at (0,0) [circle, draw, fill=black] (a1) {};
        \node at (1,0) [draw=none] (a2) {\ldots};
        \node at (2,0) [circle, draw, fill=black] (a3) {};
        \node at (3,0) [square, draw, label=above:-2] (a4) {};
        \node at (4,0) [draw=none] (a5) {\ldots};
        \node at (5,0) [square, draw, label=above:-2] (a6) {};
        \node at (6,0) [square, draw, label=above:-2] (a7) {};

        \node at (9,-2) [circle, draw, fill=black] (c-2) {};
        \node at (8,-2) [draw=none] (c-1) {\ldots};
        \node at (7,-2) [circle, draw, fill=black] (c0) {};
        \node at (6,-2) [square, draw, label=below:-(m+2)] (c1) {};
        \node at (5,-2) [square, draw, label=below:-2] (c2) {};
        \node at (4,-2) [draw=none] (c3) {\ldots};
        \node at (3,-2) [square, draw, label=below:-2] (c4) {};

        \node at (6,-1) [circle, draw] (1) {} ;

        \draw (a1) -- (a2) -- (a3) -- (a4) -- (a5) -- (a6) -- (a7);
        \draw (c-2) -- (c-1) -- (c0) -- (c1) -- (c2) -- (c3) -- (c4);
        \draw (a7) -- (1) -- (c1);

        \draw[decoration={brace,raise=20pt},decorate]
  (3,0) -- node[above=20pt] {$m$} (6,0);
        \draw[decoration={brace, mirror, raise=20pt},decorate]
  (3,-2) -- node[below=20pt] {$n$} (5,-2);

    \end{tikzpicture}
}

 \caption{The T.2.3 graph} \label{fT.2.3}
\end{figure}
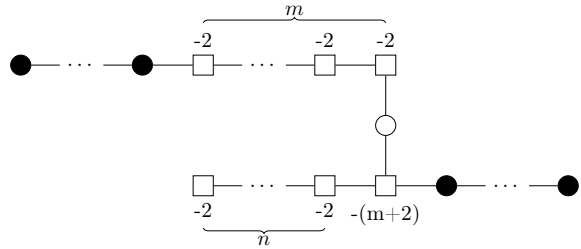
\label{T.2.3}
\end{example}


\begin{example}
\textbf{(T.2.4)} In this case, $G_{E_i}$ is a tree with two T-chains and one $(-1)$-curve connecting them as shown in Figure~\ref{fT.2.45}. The self-intersections of the curves in $E_i$ are shown in the figure. 

\begin{figure}[htbp]
\scalebox{0.8}{
\centering
    \begin{tikzpicture}[square/.style={regular polygon,regular polygon sides=4},scale=1]

        \node at (2,0) [circle, draw, fill=black] (a1) {};
        \node at (3,0) [draw=none] (a2) {\ldots};
        \node at (4,0) [circle, draw, fill=black] (a3) {};
        \node at (5,0) [square, draw, label=above:-2] (a4) {};
        \node at (6,0) [square, draw, label=above:-2] (a5) {};

        \node at (5,-2) [square,draw, label=below:-4] (b1) {};

        \node at (5,-1) [circle, draw] (1) {} ;

        \draw (a1) --(a2) -- (a3) -- (a4) -- (a5);
        \draw (a4) -- (1) -- (b1);

    \end{tikzpicture}

    \hspace{1cm}

 \begin{tikzpicture}[square/.style={regular polygon,regular polygon sides=4},scale=1]

        \node at (1,0) [circle, draw, fill=black] (a1) {};
        \node at (2,0) [draw=none] (a2) {\ldots};
        \node at (3,0) [circle, draw, fill=black] (a3) {};
        \node at (4,0) [square, draw, label=above:-2] (a4) {};
        \node at (5,0) [square, draw, label=above:-2] (a5) {};
        \node at (6,0) [square, draw, label=above:-2] (a6) {};

        \node at (7,-2) [square,draw, label=below:-3] (b1) {};
        \node at (6,-2) [square, draw,label=below:-5] (b2) {};
        \node at (5,-2) [square, draw, label=below:-2] (b3) {};

        \node at (6,-1) [circle, draw] (1) {} ;

        \draw (a1) --(a2) -- (a3) -- (a4) -- (a5) -- (a6);
        \draw (b1) -- (b2) -- (b3);
        \draw (a6) -- (1) -- (b2);

    \end{tikzpicture}

}

 \caption{The T.2.4 graph (left) and the T.2.5 graph (right).} \label{fT.2.45}
\end{figure}
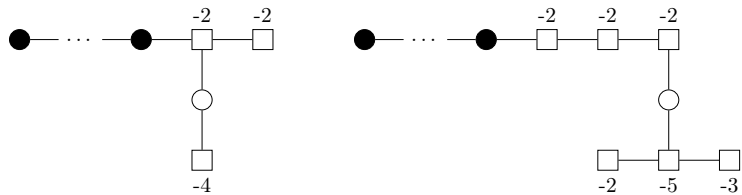
\label{T.2.4}
\end{example}


\begin{example}
\textbf{(T.2.5)} In this case, $G_{E_i}$ is a tree with two T-chains and one $(-1)$-curve connecting them as shown in Figure~\ref{fT.2.45}. The self-intersections of the curves in $E_i$ are shown in the figure.

\label{T.2.5}
\end{example}

\begin{example}
\textbf{(T.3.1)} In this case, $G_{E_i}$ is a tree with three T-chains and two $(-1)$-curves connecting them as shown in Figure~\ref{fT.3.1}. The self-intersections of the curves in $E_i$ are shown in the figure, as well as the parameters $n,m$. Both of them must be positive. To see this, observe at first that at least one, say $n$, must be positive. If $m=0$, then computing the discrepancies would give that $K_W$ is not nef. 
\label{T.3.1}
\end{example}

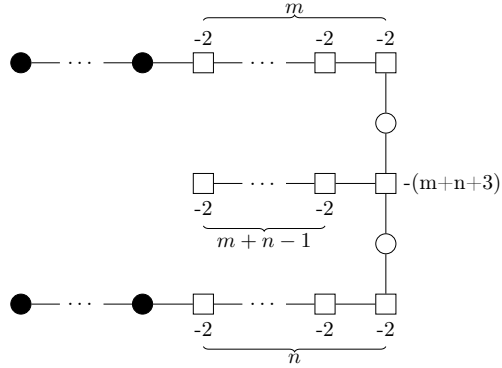
\begin{figure}[htbp]
\scalebox{0.8}{
\centering
    \begin{tikzpicture}[square/.style={regular polygon,regular polygon sides=4},scale=1]

        \node at (0,0) [circle, draw, fill=black] (a1) {};
        \node at (1,0) [draw=none] (a2) {\ldots};
        \node at (2,0) [circle, draw, fill=black] (a3) {};
        \node at (3,0) [square, draw, label=above:-2] (a4) {};
        \node at (4,0) [draw=none] (a5) {\ldots};
        \node at (5,0) [square, draw, label=above:-2] (a6) {};
        \node at (6,0) [square, draw, label=above:-2] (a7) {};

        \node at (6,-2) [square, draw, label=right:-(m+n+3)] (c1) {};
        \node at (5,-2) [square, draw, label=below:-2] (c2) {};
        \node at (4,-2) [draw=none] (c3) {\ldots};
        \node at (3,-2) [square, draw, label=below:-2] (c4) {};

        \node at (0,-4) [circle, draw, fill=black] (b1) {};
        \node at (1,-4) [draw=none] (b2) {\ldots};
        \node at (2,-4) [circle, draw, fill=black] (b3) {};
        \node at (3,-4) [square, draw, label=below:-2] (b4) {};
        \node at (4,-4) [draw=none] (b5) {\ldots};
        \node at (5,-4) [square, draw, label=below:-2] (b6) {};
        \node at (6,-4) [square, draw, label=below:-2] (b7) {};

        \node at (6,-1) [circle, draw] (1) {} ;
        \node at (6,-3) [circle, draw] (2) {} ;

        \draw (a1) -- (a2) -- (a3) -- (a4) -- (a5) -- (a6) -- (a7);
        \draw (c1) -- (c2) -- (c3) -- (c4);
        \draw (b1) -- (b2) -- (b3) -- (b4) -- (b5) -- (b6) -- (b7);
        \draw (a7) -- (1) -- (c1);
        \draw (b7) -- (2) -- (c1);

        \draw[decoration={brace,raise=20pt},decorate]
  (3,0) -- node[above=20pt] {$m$} (6,0);
        \draw[decoration={brace,mirror,raise=20pt},decorate]
  (3,-2) -- node[below=20pt] {$m+n-1$} (5,-2);
        \draw[decoration={brace, mirror, raise=20pt},decorate]
  (3,-4) -- node[below=20pt] {$n$} (6,-4);
    \end{tikzpicture}
}
    \caption{The T.3.1 graph} \label{fT.3.1}
\end{figure}

\begin{example}
\textbf{(T.3.2)} In this case, $G_{E_i}$ is a tree with three T-chains and two $(-1)$-curves connecting them as shown in Figure~\ref{fT.3.2}. The self-intersections of the curves in $E_i$ are shown in the figure, as well as the parameters $n,m>0$.

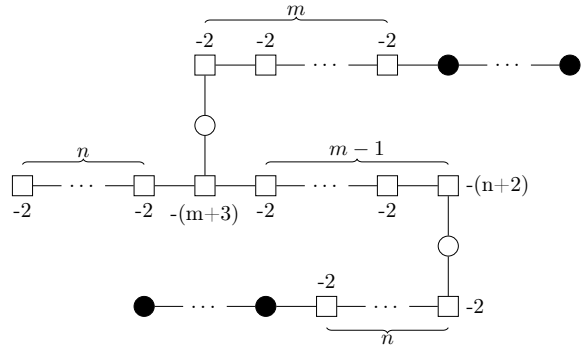
\begin{figure}[htbp]
\scalebox{0.8}{
\centering
    \begin{tikzpicture}[square/.style={regular polygon,regular polygon sides=4},scale=1]

        \node at (12,0) [circle, draw, fill=black] (a1) {};
        \node at (11,0) [draw=none] (a2) {\ldots};
        \node at (10,0) [circle, draw, fill=black] (a3) {};
        \node at (9,0) [square, draw, label=above:-2] (a4) {};
        \node at (8,0) [draw=none] (a5) {\ldots};
        \node at (7,0) [square, draw, label=above:-2] (a6) {};
        \node at (6,0) [square, draw, label=above:-2] (a7) {};

        \node at (10,-2) [square, draw, label=right:-(n+2)] (c-2) {};
        \node at (9,-2) [square, draw, label=below:-2] (c-1) {};
        \node at (8,-2) [draw=none] (c0) {\ldots};
        \node at (7,-2) [square, draw, label=below:-2] (c1) {};
        \node at (6,-2) [square, draw, label=below:-(m+3)] (c2) {};
        \node at (5,-2) [square, draw, label=below:-2] (c3) {};
        \node at (4,-2) [draw=none] (c4) {\ldots};
        \node at (3,-2) [square, draw, label=below:-2] (c5) {};

        \node at (10,-4) [square, draw, label=right:-2] (b1) {};
        \node at (9,-4) [draw=none] (b2) {\ldots};
        \node at (8,-4) [square, draw, label=above:-2] (b3) {};
        \node at (7,-4) [circle, draw, fill=black] (b4) {};
        \node at (6,-4) [draw=none] (b5) {\ldots};
        \node at (5,-4) [circle, draw, fill=black] (b6) {};

        \node at (6,-1) [circle, draw] (1) {} ;
        \node at (10,-3) [circle, draw] (2) {} ;

        \draw (a1) -- (a2) -- (a3) -- (a4) -- (a5) -- (a6) -- (a7);
        \draw (c-2) -- (c-1) -- (c0) -- (c1) -- (c2) -- (c3) -- (c4) -- (c5);
        \draw (b1) -- (b2) -- (b3) -- (b4) -- (b5) -- (b6);
        \draw (a7) -- (1) -- (c2);
        \draw (c-2) -- (2) -- (b1);

        \draw[decoration={brace,raise=20pt},decorate]
  (6,0) -- node[above=20pt] {$m$} (9,0);
        \draw[decoration={brace, raise=10pt},decorate]
  (3,-2) -- node[above=10pt] {$n$} (5,-2);
        \draw[decoration={brace, raise=10pt},decorate]
  (7,-2) -- node[above=10pt] {$m-1$} (10,-2);
        \draw[decoration={brace,mirror, raise=10pt},decorate]
  (8,-4) -- node[below=10pt] {$n$} (10,-4);
    \end{tikzpicture}
}
    \caption{The T.3.2 graph} \label{fT.3.2}

\end{figure}
\label{T.3.2}
\end{example}



        




\begin{example}
\textbf{(C.1)} In this case, $G_{E_i}$ is a cycle with one T-chain and one $(-1)$-curve as shown in Figure~\ref{fC.12} left. The self-intersections of the curves in $E_i$ are shown in the figure, as well as the parameter $m>0$.

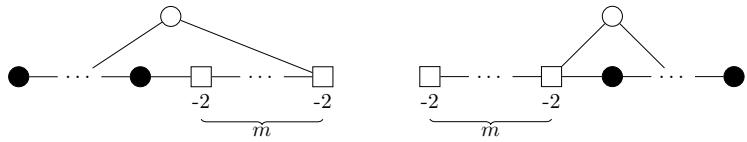
\begin{figure}[htbp]
\scalebox{0.8}{
\centering
    \begin{tikzpicture}[square/.style={regular polygon,regular polygon sides=4},scale=1]

        \node at (0,0) [circle, draw, fill=black] (a1) {};
        \node at (1,0) [draw=none] (a2) {\ldots};
        \node at (2,0) [circle, draw, fill=black] (a3) {};
        \node at (3,0) [square, draw,label=below:-2] (a4) {};
        \node at (4,0) [draw=none] (a5) {\ldots};
        \node at (5,0) [square, draw,label=below:-2] (a6) {};

        \node at (2.5,1) [circle, draw] (1) {};

        \draw (a1) -- (a2) -- (a3) -- (a4) -- (a5) -- (a6);
        \draw (a2) -- (1) -- (a6);

        \draw[decoration={brace, mirror, raise=20pt},decorate]
  (3,0) -- node[below=20pt] {$m$} (5,0);
    \end{tikzpicture}

\hspace{1cm}
 \begin{tikzpicture}[square/.style={regular polygon,regular polygon sides=4},scale=1]
       
        \node at (8,-2) [circle, draw, fill=black] (c-2) {};
        \node at (7,-2) [draw=none] (c-1) {\ldots};
        \node at (6,-2) [circle, draw, fill=black] (c0) {};
        \node at (5,-2) [square, draw, label=below:-2] (c1) {};
        \node at (4,-2) [draw=none] (c3) {\ldots};
        \node at (3,-2) [square, draw, label=below:-2] (c4) {};
        
        \node at (6,-1) [circle, draw] (1) {} ;

        \draw (c-2) -- (c-1) -- (c0) -- (c1) -- (c3) -- (c4);
        \draw (c-1) -- (1) -- (c1);

        \draw[decoration={brace, mirror, raise=20pt},decorate]
  (3,-2) -- node[below=20pt] {$m$} (5,-2);

    \end{tikzpicture}
    }

\caption{The C.1 graph (left) and the C.2 graph (right)} \label{fC.12}

\end{figure}
\label{C.1}
\end{example}


\begin{example}
\textbf{(C.2)} In this case, $G_{E_i}$ is a cycle with one T-chain and one $(-1)$-curve as shown in Figure~\ref{fC.12} right. The self-intersections of the curves in $E_i$ are shown in the figure, as well as the parameter $m>0$.
\label{C.2}
\end{example}


At the beginning, we will assume that the simple graph $G_{E_i}$ is a tree to facilitate our analysis. At the end, we will show how to classify all the cases when $G_{E_i}$ is not a tree via a suitable combinatorial reduction to the case of a tree. 

\begin{lemma}
If $G_{E_i}$ is a tree, then every $(-1)$-curve in $E_i$ intersects exactly two T-chains.
\label{numb-1}
\end{lemma}

\begin{proof}
Let $A$ be a $(-1)$-curve in $E_i$. As $K_W$ is ample, $A$ must intersect $C$ at least twice. The curve $A$ cannot intersect a T-chain at more than one point, because $G_{E_i}$ is a tree. If $A$ intersects $3$ or more T-chains, then we have four cases:

$\bullet$ The curve $A$ intersects three T-chains at three distinct curves in $E_i$. But, as noted in Remark~\ref{-1 curve}, this is not possible as such an $A$ would produce a triple point after being contracted.

$\bullet$ The curve $A$ intersects two T-chains at two curves in $E_i$, and one T-chain at a curve not in $E_i$. Then $A$ has multiplicity $2$ in $E_i$, and no other curve in $E_i$ intersects a curve in a T-chain not in $E_i$, because $E_i \cdot (\sum_{C_{k,j} \nsubseteq E_i} C_{k,j} ) =2$. The curve $A$ has two neighboring curves $X$ and $Y$ in $E_i$, each of which has multiplicity $1$ in $E_i$. Note that the T-chains corresponding to $X$ and $Y$ are completely contained in $E_i$, since otherwise we would have $E_i \cdot C >2$.


If $X$ or $Y$ is a $(-2)$-curve, say $X$, then the multiplicity of $X$ is at least 2, which implies that $A$ has multiplicity at least $3$, a contradiction. So, $X$ and $Y$ are not $(-2)$-curves.

We now must have two trees as components of $E_i \setminus A$, one for $X$ and one for $Y$. One of these two trees must contract to a $(-2)$-curve before we contract $A$, so that the complete contraction of $E_i$ can happen. The tree that contracts to a $(-2)$-curve cannot be a path graph since that would contradict Corollary~\ref{starc}. Hence we have a vertex with three neighbors, whose multiplicity is at least 2, by Lemma~\ref{3Neighbors}, hence $E_i \cdot (\sum_{C_{k,j} \nsubseteq E_i} C_{k,j} ) \geq 3$, a contradiction. 
    


$\bullet$ The curve $A$ intersects one T-chain at one curve in $E_i$, and two T-chains at two curves not in $E_i$. Then $A$ has multiplicity $1$ in $E_i$, and no other curve in $E_i$ intersects a curve in a T-chain not in $E_i$, because $E_i \cdot (\sum_{C_{k,j} \nsubseteq E_i} C_{k,j} ) =2$. Hence $A$ is a leaf of the tree $E_i$, and there is at least one T-chain in $E_i$. If $E_i$ is a path graph, then applying Corollary~\ref{starc} leads to a contradiction. So $E_i$ is not a path graph, and we can apply Lemma~\ref{3Neighbors} to a vertex with $3$ or more neighbors to again arrive at a contradiction as in the previous point. 

$\bullet$ The curve $A$ intersects three T-chains at curves not in $E_i$, then $$E_i \cdot (\sum_{C_{k,j} \nsubseteq E_i} C_{k,j} ) \geq 2,$$ a contradiction.

Therefore a $(-1)$-curve in $E_i$ intersects exactly two T-chains. 
\end{proof}


\begin{proposition}
If $G_{E_i}$ is a tree and every vertex in $E_i$ has at most two neighbours in $G_{E_i}$, then $G_{E_i}$ corresponds to Figure~\ref{fT.2.1} or Figure~\ref{fT.2.2}.
\label{no3}
\end{proposition}

\begin{proof}
Let $A$ be a $(-1)$-curve in $E_i$. By 
Lemma~\ref{numb-1}, the curve $A$ intersects exactly two curves in distinct T-chains. There are three possibilities for how $A$ intersects these two T-chains.

\textbf{(Case 1):} The vertex $A$ is adjacent to two black dots. Then, by hypothesis, $E_i=A$. But then there is no curve in $C$ inside of $E_i$, and such an $E_i$ does not satisfy our assumption that $E_i \cdot \Big(\sum_{C_{k,j} \subseteq E_i} C_{k,j} \Big) =-1$. 

\textbf{(Case 2):} The vertex $A$ is adjacent to one black dot and one square. The square neighbor of $A$ is adjacent to at most two vertices in $G_{E_i}$ by hypothesis, and so it must be at the end of the T-chain. This square must be a $(-2)$-curve and all the other squares after that, since otherwise we would need an extra $(-1)$-curve in $E_i$. But such a curve would mean that we contract a curve in the center, contradicting Corollary~\ref{starc}. Thus $G_{E_i}$ is as in Figure~\ref{fT.2.1}.

\textbf{(Case 3):} Two squares as neighbors of $A$. Both belong to two distinct T-chains, none of them inside of $E_i$ since otherwise we could apply Corollary~\ref{starc} to obtain a contradiction. Therefore both T-chains have curves out of $E_i$, and thus $G_{E_i}$ is as in  Figure~\ref{fT.2.2}. 
\end{proof}

\begin{proposition}
If $G_{E_i}$ is a tree, then every vertex in $E_i$ has at most three neighbours in $G_{E_i}$.
\label{cchain}
\end{proposition}

\begin{proof}
Suppose there is a vertex $A$ in $E_i$ with more than three neighbours in $G_{E_i}$. By Lemma~\ref{3Neighbors}, no  connected component of $G_{E_i}\setminus A$ is fully contracted before $A$. 
So, contract curves until $A$ becomes a $(-1)$-curve, and with some abuse of notation denote by $E_i$ and $G_{E_i}$ the images of the corresponding curves.

Now observe that the $(-1)$-curve $A$ continues to have at least four neighbours in $G_{E_i}$. By Remark~\ref{-1 curve} we know that $A$ cannot have three neighbours in $E_i$. But, $A$ cannot have three neighbours in $G_{E_i}\setminus E_i$ either because that would give $E_i \cdot \Big(\sum_{C_{k,j} \nsubseteq E_i} C_{k,j} \Big) \geq 3$.

Therefore the only possibility is that $A$ has two neighbours in $E_i$, and two neighbours in $G_{E_i}\setminus E_i$. But this means $A$ has multiplicity at least 2 in the divisor $E_i$, and so by pulling-back we would have had $E_i \cdot \Big(\sum_{C_{k,j} \nsubseteq E_i} C_{k,j} \Big) \geq 4$, a contradiction.
\end{proof}

\begin{proposition}
If $G_{E_i}$ is a tree, then there is at most one vertex in $E_i$ with three neighbours in $G_{E_i}$.
\label{most1tripleV}
\end{proposition}

\begin{proof}


Suppose there is a vertex $A_1$ in $E_i$ with three neighbours. Since  $E_i \cdot \Big(\sum_{C_{k,j} \nsubseteq E_i} C_{k,j} \Big) =2$, there must be at least one component of $G_{E_i} \setminus A_1$ completely contained in $E_i$. If one such component is a path graph, then define $A$ to be the vertex $A_1$. Otherwise, let $H$ be a component of $G_{E_i} \setminus A_1$ completely contained in $E_i$ and let $A$ be a vertex in $H$ of degree three whose distance from $A_1$ is maximal. Now observe that of the three components of $G_{E_i}\setminus A$, at least one is completely contained in $E_i$, and in fact every component of $G_{E_i}\setminus A$ that is completely contained in $E_i$ is  a path graph.   We now have two cases:


\bigskip 
\textbf{(Case 1):} Suppose that the remaining two components of $G_{E_i} \setminus A$ have curves outside of $E_i$, and then suppose for a contradiction that there exists a degree 3 vertex $B\in E_i$ lying in one of these two components.  Since  $E_i \cdot \Big(\sum_{C_{k,j} \nsubseteq E_i} C_{k,j} \Big) =2$, there must be at least one component of $G_{E_i} \setminus B$ completely contained in $E_i$, and the other component of $G_{E_i}\setminus B$ not containing $A$ has some curves not in $E_i$. We can assume the $B$ is contracted before $A$, as the other case is analogous by switching the roles of $A$ and $B$. We blow-down until $B$ becomes a $(-1)$-curve, which by Remark~\ref{-1 curve} cannot be connected to three curves in $E_i$.



By Lemma\ref{3Neighbors} and the assumption that $A$ gets contracted after $B$, this $(-1)$-curve $B$ is connected to two curves in $E_i$ and a curve in $C$ but not in $E_i$. But then $B$ has multiplicity at least $2$ in the image of $E_i$ on $S$. Because of the other component of $G_{E_i}\setminus A$, we have $E_i \cdot \Big(\sum_{C_{k,j} \nsubseteq E_i} C_{k,j} \Big) \geq 3$, a contradiction.



\bigskip 
\textbf{(Case 2):} Exactly one component of $G_{E_i} \setminus A$ has curves outside of $E_i$.  No component of $G_{E_i}\setminus A$ can be fully bloww down before $A$ becomes a $(-1)$-curve, by Lemma~\ref{3Neighbors}.  We do blow-downs until $A$ becomes a $(-1)$-curve. By Remark~\ref{-1 curve}, $A$ cannot be connected to three curves in $E_i$, so it must be connected to one curve not in $E_i$ and two curves in $E_i$. Hence, we have that the $(-1)$-curve $A$ in the divisor $E_i$ has multiplicity at least $2$, and so do all curves in $E_i$ in the component of $G_{E_i}\setminus A$ containing curves outside of $E_i$. 

Now, suppose for a contradiction that there is another vertex $B$ with three neighbours. As discussed earlier in the proof, $B$ must be in the component of $G_{E_i}\setminus A$ containing curves that are not in $E_i$. Then again by Lemma~\ref{3Neighbors}, 
no component of $G_{E_i}\setminus B$ can be fully contracted before $B$ becomes a $(-1)$-curve. We blow-down until $B$ becomes a $(-1)$-curve, which is connected to: a component containing $A$, a component completely contained in $E_i$, and a component containing curves that are not in $E_i$. Since a $(-1)$-curve cannot be connected to three curves in $E_i$, $B$ is either connected to two curves in $E_i$ and a curve not in $E_i$, or it is connected to one curve in $E_i$ and two curves not in $E_i$. 

In the first case, the $(-1)$-curve $B$ in the blow down of the divisor $E_i$ has multiplicity at least $4$ (since each of the curves in $E_i$ connected to the $(-1)$-curve $B$ have multiplicity at least $2$). Therefore by pulling-back, we obtain $E_i \cdot \Big(\sum_{C_{k,j} \nsubseteq E_i} C_{k,j} \Big) \geq 4$, a contradiction.

In the second case, the $(-1)$-curve $B$ in the blow down of the divisor $E_i$ has multiplicity at least $2$ and is connected to two curves outside of $E_i$, hence by pulling-back this data, we obtain $E_i \cdot \Big(\sum_{C_{k,j} \nsubseteq E_i} C_{k,j} \Big) \geq 4$, a contradiction. 
\end{proof}


Now let us consider all the possible cases with one vertex in $E_i$ having three neighbours.

\begin{proposition}
If $G_{E_i}$ is a tree containing a vertex in $E_i$ of degree 3, then $G_{E_i}$ corresponds to one of the six graphs in Figures~\ref{fT.2.3} to~\ref{fT.3.2}. 
\label{ctree}
\end{proposition}

\begin{proof}
Let us denote this special vertex by $V$. By Lemma~\ref{3Neighbors} no component of $G_{E_i}\setminus V$ can be completely contracted before $V$ is; hence when $V$ becomes a $(-1)$-curve it must still have three neighbours.

We divide the argument into two cases, depending on how many components of $G_{E_i} \setminus V$ are fully contained in $E_i$. 

\vspace{0.3cm}
\noindent
\textbf{(Case 1):} Only one component $H \subset G_{E_i} \setminus V$ is fully contained in $E_i$. 

Observe that by Proposition~\ref{most1tripleV}, $H$ is a path graph. 

\textbf{Claim.} We claim that $H$ consists of $(-2)$-curves.

We can contract until $V$ becomes a $(-1)$-curve. By Remark~\ref{-1 curve} it cannot have 3 neighbours in $E_i$. If it had two neighbors in $E_i$, then $V$ and all the curves in $E_i$ in one of the components of $G_{E_i} \setminus V$ would have multiplicity at least $2$ in $E_i$. Therefore we would have $E_i \cdot \Big(\sum_{C_{k,j} \nsubseteq E_i} C_{k,j} \Big) \geq 3$, 
a contradiction.
So, when $V$ becomes a $(-1)$-curve, it must have only one neighbour in $E_i$.

Now we do all possible contractions in $E_i$ except for contracting the $(-1)$-curve that $V$ becomes. The image of $G_{E_i}$ consists of exactly one $(-1)$-curve (namely, the image of $V$), which must be connected to two curves not in (the blow-down of) $E_i$. The third component of the image of $G_{E_i} \setminus V$ is nothing but the image of $H$, which must continue to be a path graph in $E_i$. For this chain to be contracted, all curves in it must be $(-2)$-curves. 

Suppose now that there is a $(-1)$-curve in $H$ to begin with. Then one of the two T-chains adjacent to this curve is completely contained in $E_i$, contradicting Corollary~\ref{starc}. Therefore $H$ is indeed a chain of (-2)-curves, proving the claim.

Next, observe that $H$ is completely contained in some T-chain, and therefore $V$ is in that same T-chain. So let's consider the path graph $H\cup V$. We have two sub-cases:

\textbf{(Case 1.1):} The curves $H\cup V$ form a complete T-chain. In both of the other two connected components of $G_{E_i} \setminus V$, there must be curves in $E_i$. We showed above that all of these curves must be contracted before $V$. By Corollary~\ref{starc} there cannot be a contracted center divisor, so there can be curves from at most one T-chain in each component. Then because $V$ is the only vertex in $E_i$ of degree $3$, the vertex $V$ must be connected to each of these components by $(-1)$-curves. Then this case is T.3.1, shown in Figure~\ref{fT.3.1}.

\textbf{(Case 1.2):} The curves $H\cup V$ do not form a complete T-chain.
Let $V'$ be the other vertex adjacent to $V$ in the T-chain. Then we have two options: $V'$ is either a black dot or a square.  

If $V'$ is a black dot, then $V$ is connected to another T-chain via a $(-1)$-curve. By Corollary~\ref{starc}, that T-chain is the only T-chain in its  component of $G_{E_i} \setminus V$. Since all the curves in $E_i$ in this T-chain are contracted before $V$, they must only be $-2$-curves.
Then, this case is T.2.3, shown in Figure~\ref{fT.2.3}. 

If $V'$ is a square, then we can contract until $V$ becomes a $(-1)$-curve. Suppose this $(-1)$-curve has some neighbor in $E_i$, besides the one in $H$. Then the curve $V$ and all curves in the third component of $G_{E_i}\setminus V$ (containing neither $H$ nor $V'$) would have multiplicity at least 2, contradicting the fact that $E_i\cdot (\sum_{C_{k,j}\not\subseteq E_i})=2$. Hence there must be a $(-1)$-curve in the component of $G_{E_i} \setminus V$ containing $V'$. Therefore in this component, we must have another T-chain.
By Corollary~\ref{starc}, there is only one more T-chain in this component. And by assumption, this other T-chain contains black dots.

The situation is shown in Figure~\ref{fcase 1.2}, where $n>0$, $m\geq 0$. 
We know that $V$ will become a $(-1)$-curve after the contraction of the curves above $V$ and to the right of $V$, and in that moment we will have a black dot over $V$ and another to the right of $V$. In this way, by Corollary~\ref{starc}, the center of the T-chain containing $V$ must be $V$.  Therefore all the self-intersections of the squares are determined, and this case is T.3.2, represented in Figure~\ref{fT.3.2}. In particular $m>0$.

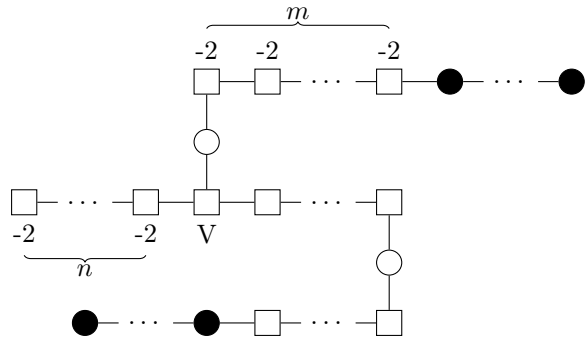
\begin{figure}[htbp]
\centering
    \begin{tikzpicture}[square/.style={regular polygon,regular polygon sides=4},scale=0.8]

        \node at (12,0) [circle, draw, fill=black] (a1) {};
        \node at (11,0) [draw=none] (a2) {\ldots};
        \node at (10,0) [circle, draw, fill=black] (a3) {};
        \node at (9,0) [square, draw, label=above:-2] (a4) {};
        \node at (8,0) [draw=none] (a5) {\ldots};
        \node at (7,0) [square, draw, label=above:-2] (a6) {};
        \node at (6,0) [square, draw, label=above:-2] (a7) {};

        \node at (9,-2) [square, draw] (c-2) {};
        \node at (8,-2) [draw=none] (c-1) {\ldots};
        \node at (7,-2) [square, draw] (c0) {};
        \node at (6,-2) [square, draw, label=below:V] (c1) {};
        \node at (5,-2) [square, draw, label=below:-2] (c2) {};
        \node at (4,-2) [draw=none] (c3) {\ldots};
        \node at (3,-2) [square, draw, label=below:-2] (c4) {};

        \node at (9,-4) [square, draw] (b1) {};
        \node at (8,-4) [draw=none] (b2) {\ldots};
        \node at (7,-4) [square, draw] (b3) {};
        \node at (6,-4) [circle, draw, fill=black] (b4) {};
        \node at (5,-4) [draw=none] (b5) {\ldots};
        \node at (4,-4) [circle, draw, fill=black] (b6) {};

        \node at (6,-1) [circle, draw] (1) {} ;
        \node at (9,-3) [circle, draw] (2) {} ;

        \draw (a1) -- (a2) -- (a3) -- (a4) -- (a5) -- (a6) -- (a7);
        \draw (c-2) -- (c-1) -- (c0) -- (c1) -- (c2) -- (c3) -- (c4);
        \draw (b1) -- (b2) -- (b3) -- (b4) -- (b5) -- (b6);
        \draw (a7) -- (1) -- (c1);
        \draw (c-2) -- (2) -- (b1);

        \draw[decoration={brace,raise=20pt},decorate]
  (6,0) -- node[above=20pt] {$m$} (9,0);
        \draw[decoration={brace, mirror, raise=20pt},decorate]
  (3,-2) -- node[below=20pt] {$n$} (5,-2);
    \end{tikzpicture}
    \caption{Case 1.2} \label{fcase 1.2}
\end{figure}

\noindent
\textbf{(Case 2):} Exactly two components  of $G_{E_i} \setminus V$ are fully contained in $E_i$.


We blow down curves in $E_i$ as much as possible until the image of $V$ is the only $(-1)$-curve. We know by assumption and by Remark~\ref{-1 curve} that $V$ is connected to one curve not in the blow-down of $E_i$, and two chains of curves $H_1$ and $H_2$ contained in the blow-down of $E_i$. 

Observe now that for both components to be blown down further, there needs to be a $(-2)$-curve intersecting the blow-down of $V$; call it $V'$, and suppose without loss of generality that $V'\in H_1$.

We have the following observations:

$\bullet$ \textit{$H_1=V'$ consists of exactly one $(-2)$-curve.} Indeed, suppose for a contradiction that $V'$ has another neighbor in $H_1$. Then the image of $V$ has multiplicity at least $3$ in the image of $E_i$, and so by pulling-back we obtain $E_i \cdot \Big(\sum_{C_{k,j} \nsubseteq E_i} C_{k,j} \Big) \geq 3$, a contradiction.

$\bullet$ \textit{$H_2$ is either $[3,2,\ldots,2]$ or $[3]$}. This is just because otherwise we cannot completely contract $E_i$ to a smooth point.   

$\bullet$\textit{The component of $G_{E_i} \setminus V$ that contracts to $H_1=V'$ is also a $(-2)$-curve}. Indeed, suppose not. Then this means the component must have originally contained another $(-1)$-curve, and so a T-chain whose center is contracted. But this is not possible by Corollary~\ref{starc}.

$\bullet$ \textit{The curve $V$ is part of a T-chain}. This is because $V'$ must intersect a T-chain (or be inside one), and a $(-2)$-curve cannot be a T-chain.

$\bullet$ \textit{The curve $V'$ must be part of a T-chain}. Otherwise, if $V'$ were not part of a T-chain, then $G_{E_i}\setminus V'$ is exactly the situation of Corollary~\ref{starc} part (a). In particular, contracting all possible curves in $G_{E_i}\setminus V'$, we would have a curve in $H_2$ of self-intersection less than $(-2)$, contradicting our above argument.

By abuse of notation, from now on we refer to the components of $G_{E_i}\setminus V$ contained in $E_i$ as $H_1=V'$ and $H_2$.

\bigskip 
\textbf{(Case 2.1):}  Assume that $H_2$  does not contain a $(-1)$-curve. 

We claim that $V' \cup V \cup H_2$ forms a T-chain or a T-chain with one $(-2)$-curve at the end of $H_2$ that is not in a T-chain. 

Suppose for a contradiction that the T-chain containing $V$ contains no curve in $H_2$. 
Then $G_{E_i} \setminus H_2$ is the situation of Corollary~\ref{starc}. Contracting all possible curves here, we would contract a center, a contradiction.

Since $V'$ is a $(-2)$ curve and $H_2$ is as described above, the T-chain containing $V$ must be of the form $[2,X,3]$, hence $[2,5,3]$. By Lemma~\ref{disc}, we can have at most one additional $(-2)$-curve in $H_2$. This gives us the case T.2.5, the situation shown in Figure~\ref{fT.2.45} right.

\bigskip 
\textbf{(Case 2.2):} 
Assume that $H_2$ contains at least one $(-1)$-curve. 

We claim that $H_2$ does not contain any curves in the T-chain of $V$. Assume for a contradiction that $H_2$ contains curves of the T-chain of $V$. Then $V\cup V'\cup H_2$ is a chain as in Corollary~\ref{starc}. Contracting all possible $(-1)$-curves in this chain, and using that the image of $H_2$ must contain only one $(-3)$-curve, we see that $V$ must be the center of its T-chain, because the remaining curves in that T-chain (apart from $V$ and $V'$) are contracted. Moreover, the lack of additional curves between the image of $V$ and the $(-3)$-curve in the image of $H_2$ implies that there is only one $(-1)$-curve in $H_2$.

Therefore the T-chain containing $V$ must be of the form $$[2,X,2,\ldots,2,3],$$ which is then connected to a $(-1)$-curve at the ending $(-3)$-curve. This $(-1)$-curve must also intersect a second T-chain, and thus, in order for $H_2$ to be contracted to a $(-3)$-curve, the remainder of $H_2$ must be $[2,Y]$, therefore $[2,5]$. This implies that $V'\cup V\cup H_2$ is precisely $[2,5,3]-1-[2,5]$. But then the $(-1)$-curve intersects curves with discrepancies adding up to  $-\frac{1}{3}-\frac{3}{5}=-\frac{14}{15}> -1$, contradicting Lemma~\ref{(-1)-(-2)-curve}. This concludes the proof of the claim.

Therefore the T-chain of $V$ is either $[2,5]$ or it contains curves in the third component of $G_{E_i}\setminus V$. 

We can now determine $H_2$; it must have exactly one $(-1)$-curve, which must be connected to $V$ and a curve in the T-chain in $H_2$. Since $H_2$ contracts to $[3,2,\ldots, 2]$, the remainder of $H_2$ must be some part of $[2,\ldots,2,4,2,\ldots,2]$. The only T-chain that can be contained here is $[4]$, hence $H_2$ is a $(-1)$-curve connected to a T-chain $[4]$ and possibly a $(-2)$-curve not in a T-chain. 

If the T-chain of $V$ contains more curves than $V$ and $V'$, then in this third component there cannot be any $(-1)$-curve, since otherwise, doing all the contractions except for the one $(-1)$-curve in $H_2$, the T-chain containing $V$ ends up contracted into only two $(-2)$-curves, contradicting Corollary~\ref{starc}. In particular, this implies that $V$ is a $(-2)$-curve, and by Remark~\ref{-1 curve}, the third component $G_{E_i}\setminus (V'\cup V\cup H_2)$ cannot have curves in $E_i$. Hence $E_i$ is $V$, $V'$ and $H_2$.
This gives us case T.2.4, displayed in Figure 
\ref{fT.2.45} left. 

Finally, if the T-chain containing $V$ is $[5,2]$, then in order to be able to contract $V$, the component $E_i \setminus (V\cup V'\cup H_2)$ must be $[2,2,1]$ with the $(-1)$-curve attached to $V$, because $H_2$ is $[1,4]$. So, the situation for $G_{E_i}$ is one T-chain $[A,...,B,2,2]$ followed by a $(-1)$-curve $C$ connecting the last $(-2)$-curve with the $(-5)$-curve in $[5,2]$, and then followed by another $(-1)$-curve connecting the $(-5)$-curve to the $[4]$. We claim that in this situation $C \cdot K_W <0$, and so this case in impossible. We follow the notation in Proposition~\ref{T-chain} part (iv).    
Say that $t_1$ for $A$ is $a$ and $t_s$ for $B$ is $b$. Then $t_{s+1}=a+b$ and $t_{s+2}=2a+b$, and the index for the T-chain is $3a+b$. Therefore the discrepancy for the last $(-2)$-curve is $-a/(3a+b) > -1/3$. On the other hand, the discrepancy of the $(-5)$-curve in $[5,2]$ is $-2/3$ and so $2/3+ a/(3a+b) <1$, and this implies $C \cdot K_W <0$.




\end{proof}

\begin{figure}[htbp]
\centering
    \begin{tikzpicture}[square/.style={regular polygon,regular polygon sides=4},scale=0.8]

        \node at (0,0) [circle, draw, fill=black] (a1) {};
        \node at (1,0) [draw=none] (a2) {\ldots};
        \node at (2,0) [circle, draw, fill=black] (a3) {};
        \draw (a1) -- (a2) -- (a3);

    \end{tikzpicture}
    \caption{Case 1} \label{f1}
 \vspace{0.3cm}

\centering
    \begin{tikzpicture}[square/.style={regular polygon,regular polygon sides=4},scale=0.8]

        \node at (0,0) [circle, draw, fill=black] (a1) {};
        \node at (1,0) [draw=none] (a2) {\ldots};
        \node at (2,0) [circle, draw, fill=black] (a3) {};
        \node at (3,0) [square, draw] (a4) {};
        \node at (4,0) [draw=none] (a5) {\ldots};
        \node at (5,0) [square, draw] (a6) {};
        \draw (a1) -- (a2) -- (a3) -- (a4) -- (a5) -- (a6);

    \end{tikzpicture}
    \caption{Case 2} \label{f2}
\vspace{0.3cm}

\centering
    \begin{tikzpicture}[square/.style={regular polygon,regular polygon sides=4},scale=0.8]

        \node at (0,0) [square, draw] (a1) {};
        \node at (1,0) [draw=none] (a2) {\ldots};
        \node at (2,0) [square, draw] (a3) {};
        \node at (3,0) [circle, draw, fill=black] (a4) {};
        \node at (4,0) [draw=none] (a5) {\ldots};
        \node at (5,0) [circle, draw, fill=black] (a6) {};
        \node at (6,0) [square, draw] (a7) {};
        \node at (7,0) [draw=none] (a8) {\ldots};
        \node at (8,0) [square, draw] (a9) {};

        \draw (a1) -- (a2) -- (a3) -- (a4) -- (a5) -- (a6) -- (a7) -- (a8) -- (a9);

    \end{tikzpicture}
    \caption{Case 3} \label{f3}
\end{figure}

We now explain how the classification for the cases when $G_{E_i}$ is a tree gives a classification for the cases when $G_{E_i}$ contains a cycle.

A cycle in $G_{E_i}$ must contain both white and black vertices. Hence a cycle must contain at least 2 edges between black and white vertices. Therefore a cycle must contain exactly 2 edges between black and white vertices, because we are assuming $E_i \cdot C=1$. Therefore, if we consider only the black vertices in $G_{E_i}$, they have only one connected component.

Hence, the set of all T-chains in $G_{E_i}$ which are not completely contained in $E_i$ must correspond to one of the three cases in Figures~\ref{f1}--\ref{f3}.


Observe that in each case, we have only one T-chain not fully contained in  $E_i$. So for each case, we assume the example is realizable, and construct a related combinatorial example containing two T-chains not fully contained in $E_i$. The key is that the new $G'_{E_i}$ is a tree and satisfies the same combinatorial constraints as $G_{E_i}$.

\vspace{0.3cm}

\textbf{(Case 1):} We construct  a new graph $G'_{E_i}$ in the following way. In $G_{E_i}$ there are one or two curves in $E_i$ connected to the T-chain which is not contained in $E_i$. Disconnect one of these two intersections and reconnect it, instead, to the corresponding vertex in a new equal T-chain. Then $G'_{E_i}$ is a tree. Moreover, it fulfills the same combinatorial constraints as $G_{E_i}$, i.e. $E_i$ is the same, the graph $G'_{E_i}$ satisfies Proposition~\ref{Ei}, and $E_i$ intersects curves in the T-chains with the same discrepancies.  In this way, our new $G'_{E_i}$ has already been classified, but there are no such graphs with two T-chains completely out of $E_i$. 

\vspace{0.2cm}

\textbf{(Case 2):} We again construct in a combinatorial way a new graph $G'_{E_i}$. In $G_{E_i}$, there is a curve in $E_i$ connected to a vertex not in $E_i$ in the T-chain not fully contained in $E_i$. Disconnect this curve to this T-chain and connect it, instead, to the corresponding vertex of a new equal T-chain. This new T-chain has only black dots. 

Then $G'_{E_i}$ is a tree. This is a graph that fulfills the same constraints as $G_{E_i}$. Now we apply our classification. The only possibilities are Figure~\ref{fT.2.1} or Figure~\ref{fT.2.3} with $m=0$, because we need a T-chain with no squares among its curves. These correspond to the cases that $G_{E_i}$ is the graph C.1, respectively C.2, in Figure~\ref{fC.12}.

\vspace{0.2cm}
\textbf{(Case 3):} Once more, we construct in a combinatorial way a new graph $G'_{E_i}$ in the following way. We change the T-chain not contained in $E_i$ by two equal T-chains, changing from Figure~\ref{mod1} to Figure~\ref{mod2}.

\vspace{0.1cm}

\begin{figure}[htbp]
\centering
    \begin{tikzpicture}[square/.style={regular polygon,regular polygon sides=4},scale=0.6]

        \node at (0,0) [square, draw] (a1) {};
        \node at (1,0) [draw=none] (a2) {\ldots};
        \node at (2,0) [square, draw] (a3) {};
        \node at (3,0) [circle, draw, fill=black] (a4) {};
        \node at (4,0) [draw=none] (a5) {\ldots};
        \node at (5,0) [circle, draw, fill=black] (a6) {};
        \node at (6,0) [square, draw] (a7) {};
        \node at (7,0) [draw=none] (a8) {\ldots};
        \node at (8,0) [square, draw] (a9) {};

        \draw (a1) -- (a2) -- (a3) -- (a4) -- (a5) -- (a6) -- (a7) -- (a8) -- (a9);

        \draw[decoration={brace,mirror,raise=10pt},decorate]
  (0,0) -- node[below=10pt] {$P$} (2,0);
        \draw[decoration={brace,mirror,raise=10pt},decorate]
  (6,0) -- node[below=10pt] {$Q$} (8,0);
    \end{tikzpicture}
    \caption{} \label{mod1}
\end{figure}

\begin{figure}[htbp]
    \begin{tikzpicture}[square/.style={regular polygon,regular polygon sides=4},scale=0.6]

        \node at (0,0) [square, draw] (a1) {};
        \node at (1,0) [draw=none] (a2) {\ldots};
        \node at (2,0) [square, draw] (a3) {};
        \node at (3,0) [circle, draw, fill=black] (a4) {};
        \node at (4,0) [draw=none] (a5) {\ldots};
        \node at (5,0) [circle, draw, fill=black] (a6) {};
        \node at (6,0) [circle, draw, fill=black] (a7) {};
        \node at (7,0) [draw=none] (a8) {\ldots};
        \node at (8,0) [circle, draw, fill=black] (a9) {};

        \draw (a1) -- (a2) -- (a3) -- (a4) -- (a5) -- (a6) -- (a7) -- (a8) -- (a9);

        \node at (0,-1) [circle, draw, fill=black] (b1) {};
        \node at (1,-1) [draw=none] (b2) {\ldots};
        \node at (2,-1) [circle, draw, fill=black] (b3) {};
        \node at (3,-1) [circle, draw, fill=black] (b4) {};
        \node at (4,-1) [draw=none] (b5) {\ldots};
        \node at (5,-1) [circle, draw, fill=black] (b6) {};
        \node at (6,-1) [square, draw] (b7) {};
        \node at (7,-1) [draw=none] (b8) {\ldots};
        \node at (8,-1) [square, draw] (b9) {};

        \draw (b1) -- (b2) -- (b3) -- (b4) -- (b5) -- (b6) -- (b7) -- (b8) -- (b9);

        \draw[decoration={brace,raise=10pt},decorate]
  (0,0) -- node[above=10pt] {$P'$} (2,0);
        \draw[decoration={brace,mirror,raise=10pt},decorate]
  (6,-1) -- node[below=10pt] {$Q'$} (8,-1);

    \end{tikzpicture}
    \caption{} \label{mod2}

\end{figure}

We now connect to a curve in $P'$ (respectively $Q'$) whichever was connected to the corresponding curve in $P$ (respectively $Q$). So $G'_{E_i}$ is a tree. Again, this new graph satisfies the same constraints as $G'_{E_i}$, and so we can use our classification for this case too. In the classification above, we must have two T-chains with curves in $E_i$ and curves not in $E_i$, the only possibility is Figure~\ref{fT.2.2}. In any other case, we would have started with a T-chain with $(-2)$-curves on both ends, which is impossible. In the case of Figure~\ref{fT.2.2}, we would have started with a $(-1)$-curve connecting both ends, and this produces a zero curve in $W$ for $K_W$, which is again not possible.

We have proved the following theorem.

\begin{theorem}
If $E_i \cdot C= 1$, then $G_{E_i}$ is one of the nine graphs from Figure~\ref{fT.2.1} to Figure~\ref{fC.12}. 
\label{classificationEC=1}
\end{theorem}

For the reader's benefit, we have collected these nine graphs into the $3$ tables below.


\begin{table}
\caption{Graphs $G_{E_i}$ with $E_i \cdot C=1$ and two T-chains}
\label{table1}
\centering
\resizebox{.43\paperheight}{!}{%
\begin{tabular}{|c|c|c|}
  \hline
  Type   &  $G_{E_i}$ & Decorated graph   
  \\ \hline
  T.2.1  &
      \begin{tikzpicture}[square/.style={regular polygon,regular polygon sides=4},scale=.8]

        \node at (1,-2) [circle, draw, fill=black] (a00) {};
        \node at (2,-2) [draw=none] (a0) {\ldots};
        \node at (3,-2) [circle, draw, fill=black] (a1) {};
        \node at (4,-2) [draw=none] (a2) {\ldots};
        \node at (5,-2) [circle, draw, fill=black] (a3) {};
        \node at (3,-1) [circle, draw] (a4) {};
        \node at (3,0) [square, draw, label=left:-2] (a5) {};
        \node at (4,0) [draw=none] (a6) {\ldots};
        \node at (5,0) [square, draw, label=below:-2] (a7) {};
        \node at (6,0) [circle, draw, fill=black] (a8) {};
        \node at (7,0) [draw=none] (a9) {\ldots};
        \node at (8,0) [circle, draw, fill=black] (a10) {};

        \draw (a00) -- (a0) -- (a1) -- (a2) -- (a3);
        \draw (a5) -- (a6) -- (a7) -- (a8) -- (a9) -- (a10);
        \draw (a1) -- (a4) -- (a5);

        \draw[decoration={brace,raise=10pt},decorate]
        (3,0) -- node[above=10pt] {$m$} (5,0);

    \end{tikzpicture}

    &
    \begin{tikzpicture}[square/.style={regular polygon,regular polygon sides=4},scale=1]
\draw[opacity=0] plot coordinates {
                (0, -1)
            } ;
        \node at (0,0) [circle, draw] (a1) {};
        \node at (2,0) [circle, draw] (a2) {};
        \path[-]
        (a1) edge node[below] {$m$} (a2);

    \end{tikzpicture}
  \\ \hline 

  T.2.2
  
    &   \begin{tikzpicture}[square/.style={regular polygon,regular polygon sides=4},scale=.75]
\draw[opacity=0] plot coordinates {
                (0, 1)
            } ;

        \node at (0,0) [circle, draw, fill=black] (a1) {};
        \node at (1,0) [draw=none] (a2) {\ldots};
        \node at (2,0) [circle, draw, fill=black] (a3) {};
        \node at (3,0) [square, draw] (a4) {};
        \node at (4,0) [draw=none] (a5) {\ldots};
        \node at (5,0) [square, draw] (a6) {};
        \node at (6,0) [circle, draw] (a7) {};
        \node at (7,0) [square, draw] (a8) {};
        \node at (8,0) [draw=none] (a9) {\ldots};
        \node at (9,0) [square, draw] (a10) {};
        \node at (10,0) [circle, draw, fill=black] (a11) {};
        \node at (11,0) [draw=none] (a12) {\ldots};
        \node at (12,0) [circle, draw, fill=black] (a13) {};

        \draw (a1) -- (a2) -- (a3) -- (a4) -- (a5) -- (a6) -- (a7) -- (a8) -- (a9) -- (a10) -- (a11) -- (a12) -- (a13);

        \draw[decoration={brace, mirror,raise=10pt},decorate]
        (3,0) -- node[below=10pt] {$m$} (5,0);
        \draw[decoration={brace, mirror,raise=10pt},decorate]
        (7,0) -- node[below=10pt] {$n$} (9,0);

    \end{tikzpicture}
   &
\begin{tikzpicture}[square/.style={regular polygon,regular polygon sides=4},scale=1]

        \node at (5,0) [circle, draw] (b1) {};
        \node at (7,0) [circle, draw] (b2) {};
        \path[-]

        (b1) edge node[below] {$m+n$} (b2);
    \end{tikzpicture}

  \\ \hline
 T.2.3 
 
 & 
\begin{tikzpicture}[square/.style={regular polygon,regular polygon sides=4},scale=0.95]

        \node at (0,0) [circle, draw, fill=black] (a1) {};
        \node at (1,0) [draw=none] (a2) {\ldots};
        \node at (2,0) [circle, draw, fill=black] (a3) {};
        \node at (3,0) [square, draw, label=above:-2] (a4) {};
        \node at (4,0) [draw=none] (a5) {\ldots};
        \node at (5,0) [square, draw, label=above:-2] (a6) {};
        \node at (6,0) [square, draw, label=above:-2] (a7) {};

        \node at (9,-2) [circle, draw, fill=black] (c-2) {};
        \node at (8,-2) [draw=none] (c-1) {\ldots};
        \node at (7,-2) [circle, draw, fill=black] (c0) {};
        \node at (6,-2) [square, draw, label=below:-(m+2)] (c1) {};
        \node at (5,-2) [square, draw, label=below:-2] (c2) {};
        \node at (4,-2) [draw=none] (c3) {\ldots};
        \node at (3,-2) [square, draw, label=below:-2] (c4) {};

        \node at (6,-1) [circle, draw] (1) {} ;

        \draw (a1) -- (a2) -- (a3) -- (a4) -- (a5) -- (a6) -- (a7);
        \draw (c-2) -- (c-1) -- (c0) -- (c1) -- (c2) -- (c3) -- (c4) ;
        \draw (a7) -- (1) -- (c1);

        \node at (6,-1) [circle, draw] (1) {} ;

        \draw (a1) -- (a2) -- (a3) -- (a4) -- (a5) -- (a6) -- (a7);
        \draw (c-2) -- (c-1) -- (c0) -- (c1) -- (c2) -- (c3) -- (c4);
        \draw (a7) -- (1) -- (c1);

        \draw[decoration={brace,raise=20pt},decorate]
  (3,0) -- node[above=20pt] {$m$} (6,0);
        \draw[decoration={brace, mirror, raise=20pt},decorate]
  (3,-2) -- node[below=20pt] {$n$} (5,-2);

    \end{tikzpicture}
 
  & 

 \begin{tikzpicture}[square/.style={regular polygon,regular polygon sides=4},scale=1]
    \draw[opacity=0] plot coordinates {
                (0, -1.5)
            } ;
        \node at (0,0) [circle, draw] (a1) {};
        \node at (1,0) [circle, draw] (a2) {};
        \path[-]
        (a1) edge node[below] {$m$} (a2);
        \path[-]
        (a2) edge[loop right] node {$1$} ();
    
    \end{tikzpicture}

  \\ \hline
 T.2.4 
 
 & \begin{tikzpicture}[square/.style={regular polygon,regular polygon sides=4},scale=.75]

        \node at (2,0) [circle, draw, fill=black] (a1) {};
        \node at (3,0) [draw=none] (a2) {\ldots};
        \node at (4,0) [circle, draw, fill=black] (a3) {};
        \node at (5,0) [square, draw, label=above:-2] (a4) {};
        \node at (6,0) [square, draw, label=above:-2] (a5) {};

        \node at (5,-2) [square,draw, label=below:-4] (b1) {};

        \node at (5,-1) [circle, draw] (1) {} ;

        \draw (a1) --(a2) -- (a3) -- (a4) -- (a5);
        \draw (a4) -- (1) -- (b1);

    \end{tikzpicture}
    
    &  \begin{tikzpicture}[square/.style={regular polygon,regular polygon sides=4},scale=1]
\draw[opacity=0] plot coordinates {
                (0, -1)
            } ;
        \node at (0,0) [circle, draw] (a1) {};
        \node at (1,0) [circle, draw] (a2) {};

        \path[-]

        (a1) edge node[below] {$1$} (a2);

        \path[-]

        (a2) edge[loop right] node {$1$} ();

    \end{tikzpicture}
  \\ \hline
 T.2.5 
 
 &   \begin{tikzpicture}[square/.style={regular polygon,regular polygon sides=4},scale=.75]

        \node at (1,0) [circle, draw, fill=black] (a1) {};
        \node at (2,0) [draw=none] (a2) {\ldots};
        \node at (3,0) [circle, draw, fill=black] (a3) {};
        \node at (4,0) [square, draw, label=above:-2] (a4) {};
        \node at (5,0) [square, draw, label=above:-2] (a5) {};
        \node at (6,0) [square, draw, label=above:-2] (a6) {};

        \node at (7,-2) [square,draw, label=below:-3] (b1) {};
        \node at (6,-2) [square, draw,label=below:-5] (b2) {};
        \node at (5,-2) [square, draw, label=below:-2] (b3) {};

        \node at (6,-1) [circle, draw] (1) {} ;

        \draw (a1) --(a2) -- (a3) -- (a4) -- (a5) -- (a6);
        \draw (b1) -- (b2) -- (b3);
        \draw (a6) -- (1) -- (b2);

    \end{tikzpicture}
 
 &     \begin{tikzpicture}[square/.style={regular polygon,regular polygon sides=4},scale=1]
\draw[opacity=0] plot coordinates {
                (0, -1)
            } ;
        \node at (0,0) [circle, draw] (a1) {};
        \node at (1,0) [circle, draw] (a2) {};
        \path[-]
        (a1) edge node[below] {$3$} (a2);
        \path[-]
        (a2) edge[loop right] node {$2$} ();

    \end{tikzpicture}

 \\ \hline
  \end{tabular}}
\end{table}

\begin{table}
\caption{Graphs $G_{E_i}$ with $E_i \cdot C=1$ and three T-chains}
\label{table2}
\centering
\resizebox{.41\paperheight}{!}{%
\begin{tabular}{|c|c|c|}
  \hline
  Type   &  $G_{E_i}$ & Decorated graph   

  \\ \hline
 T.3.1
 
   &     \begin{tikzpicture}[square/.style={regular polygon,regular polygon sides=4},scale=.8]

        \node at (0,0) [circle, draw, fill=black] (a1) {};
        \node at (1,0) [draw=none] (a2) {\ldots};
        \node at (2,0) [circle, draw, fill=black] (a3) {};
        \node at (3,0) [square, draw, label=above:-2] (a4) {};
        \node at (4,0) [draw=none] (a5) {\ldots};
        \node at (5,0) [square, draw, label=above:-2] (a6) {};
        \node at (6,0) [square, draw, label=above:-2] (a7) {};

        \node at (6,-2) [square, draw, label=right:-(m+n+3)] (c1) {};
        \node at (5,-2) [square, draw, label=below:-2] (c2) {};
        \node at (4,-2) [draw=none] (c3) {\ldots};
        \node at (3,-2) [square, draw, label=below:-2] (c4) {};

        \node at (0,-4) [circle, draw, fill=black] (b1) {};
        \node at (1,-4) [draw=none] (b2) {\ldots};
        \node at (2,-4) [circle, draw, fill=black] (b3) {};
        \node at (3,-4) [square, draw, label=below:-2] (b4) {};
        \node at (4,-4) [draw=none] (b5) {\ldots};
        \node at (5,-4) [square, draw, label=below:-2] (b6) {};
        \node at (6,-4) [square, draw, label=below:-2] (b7) {};

        \node at (6,-1) [circle, draw] (1) {} ;
        \node at (6,-3) [circle, draw] (2) {} ;

        \draw (a1) -- (a2) -- (a3) -- (a4) -- (a5) -- (a6) -- (a7);
        \draw (c1) -- (c2) -- (c3) -- (c4);
        \draw (b1) -- (b2) -- (b3) -- (b4) -- (b5) -- (b6) -- (b7);
        \draw (a7) -- (1) -- (c1);
        \draw (b7) -- (2) -- (c1);

        \draw[decoration={brace,raise=20pt},decorate]
  (3,0) -- node[above=20pt] {$m$} (6,0);
        \draw[decoration={brace,mirror,raise=20pt},decorate]
  (3,-2) -- node[below=20pt] {$m+n-1$} (5,-2);
        \draw[decoration={brace, mirror, raise=20pt},decorate]
  (3,-4) -- node[below=20pt] {$n$} (6,-4);
    \end{tikzpicture}
    &
      \begin{tikzpicture}[square/.style={regular polygon,regular polygon sides=4},scale=1]
\draw[opacity=0] plot coordinates {
                (0, -2.5)
            } ;
        \node at (0,0) [circle, draw] (a1) {};
        \node at (1,0) [circle, draw] (a2) {};
        \node at (2,0) [circle, draw] (a3) {};
        \path[-]
        (a1) edge node[below] {$m$} (a2)
        (a3) edge node[below] {$n$} (a2);
        \path[-]
        (a2) edge[loop above] node {$1$} ();
    \end{tikzpicture}
  \\ \hline
 T.3.2
 
 &   \begin{tikzpicture}[square/.style={regular polygon,regular polygon sides=4},scale=.9]

        \node at (12,0) [circle, draw, fill=black] (a1) {};
        \node at (11,0) [draw=none] (a2) {\ldots};
        \node at (10,0) [circle, draw, fill=black] (a3) {};
        \node at (9,0) [square, draw, label=above:-2] (a4) {};
        \node at (8,0) [draw=none] (a5) {\ldots};
        \node at (7,0) [square, draw, label=above:-2] (a6) {};
        \node at (6,0) [square, draw, label=above:-2] (a7) {};

        \node at (10,-2) [square, draw, label=right:-(n+2)] (c-2) {};
        \node at (9,-2) [square, draw, label=below:-2] (c-1) {};
        \node at (8,-2) [draw=none] (c0) {\ldots};
        \node at (7,-2) [square, draw, label=below:-2] (c1) {};
        \node at (6,-2) [square, draw, label=below:-(m+3)] (c2) {};
        \node at (5,-2) [square, draw, label=below:-2] (c3) {};
        \node at (4,-2) [draw=none] (c4) {\ldots};
        \node at (3,-2) [square, draw, label=below:-2] (c5) {};

        \node at (10,-4) [square, draw, label=right:-2] (b1) {};
        \node at (9,-4) [draw=none] (b2) {\ldots};
        \node at (8,-4) [square, draw, label=above:-2] (b3) {};
        \node at (7,-4) [circle, draw, fill=black] (b4) {};
        \node at (6,-4) [draw=none] (b5) {\ldots};
        \node at (5,-4) [circle, draw, fill=black] (b6) {};

        \node at (6,-1) [circle, draw] (1) {} ;
        \node at (10,-3) [circle, draw] (2) {} ;

        \draw (a1) -- (a2) -- (a3) -- (a4) -- (a5) -- (a6) -- (a7);
        \draw (c-2) -- (c-1) -- (c0) -- (c1) -- (c2) -- (c3) -- (c4) -- (c5);
        \draw (b1) -- (b2) -- (b3) -- (b4) -- (b5) -- (b6);
        \draw (a7) -- (1) -- (c2);
        \draw (c-2) -- (2) -- (b1);

        \draw[decoration={brace,raise=20pt},decorate]
  (6,0) -- node[above=20pt] {$m$} (9,0);
        \draw[decoration={brace, raise=10pt},decorate]
  (3,-2) -- node[above=10pt] {$n$} (5,-2);
        \draw[decoration={brace, raise=10pt},decorate]
  (7,-2) -- node[above=10pt] {$m-1$} (10,-2);
        \draw[decoration={brace,mirror, raise=10pt},decorate]
  (8,-4) -- node[below=10pt] {$n$} (10,-4);
    \end{tikzpicture}

 &  \begin{tikzpicture}[square/.style={regular polygon,regular polygon sides=4},scale=0.8]
\draw[opacity=0] plot coordinates {
                (0, -2.5)
            } ;
        \node at (0,0) [circle, draw] (a1) {};
        \node at (2,0) [circle, draw] (a2) {};
        \node at (5,0) [circle, draw] (a3) {};
        \path[-]
        (a1) edge node[below] {$m$} (a2)
        (a3) edge node[below] {$m+n-1$} (a2);
        \path[-]
        (a2) edge[loop above] node {$1$} ();
    \end{tikzpicture}

        




    \\ \hline
  \end{tabular}}
\end{table}

\begin{table}
\caption{Graphs $G_{E_i}$ with $E_i \cdot C=1$ and one T-chain}
\label{table3}
\centering
\resizebox{.3\paperheight}{!}{%
\begin{tabular}{|c|c|c|}
  \hline
  Type   &  $G_{E_i}$ & Decorated graph

  \\ \hline
 C.1
 
 &      \begin{tikzpicture}[square/.style={regular polygon,regular polygon sides=4},scale=.7]

        \node at (0,0) [circle, draw, fill=black] (a1) {};
        \node at (1,0) [draw=none] (a2) {\ldots};
        \node at (2,0) [circle, draw, fill=black] (a3) {};
        \node at (3,0) [square, draw,label=below:-2] (a4) {};
        \node at (4,0) [draw=none] (a5) {\ldots};
        \node at (5,0) [square, draw,label=below:-2] (a6) {};

        \node at (2.5,1) [circle, draw] (1) {};

        \draw (a1) -- (a2) -- (a3) -- (a4) -- (a5) -- (a6);
        \draw (a2) -- (1) -- (a6);

        \draw[decoration={brace, mirror, raise=20pt},decorate]
  (3,0) -- node[below=20pt] {$m$} (5,0);
    \end{tikzpicture}
    
    & \begin{tikzpicture}[square/.style={regular polygon,regular polygon sides=4},scale=1]
\draw[opacity=0] plot coordinates {
                (0, -1)
            } ;
        \node at (1,0) [circle, draw] (a2) {};
        \path[-]
        (a2) edge[loop right] node {$m$} ();
    \end{tikzpicture}
  
  \\ \hline
 
 C.2
 
 &   \begin{tikzpicture}[square/.style={regular polygon,regular polygon sides=4},scale=0.7]

  \node at (8,-2) [circle, draw, fill=black] (c-2) {};
        \node at (7,-2) [draw=none] (c-1) {\ldots};
        \node at (6,-2) [circle, draw, fill=black] (c0) {};
        \node at (5,-2) [square, draw, label=below:-2] (c1) {};
        \node at (4,-2) [draw=none] (c3) {\ldots};
        \node at (3,-2) [square, draw, label=below:-2] (c4) {};
        
        \node at (6,-1) [circle, draw] (1) {} ;

        \draw (c-2) -- (c-1) -- (c0) -- (c1) -- (c3) -- (c4);
        \draw (c-1) -- (1) -- (c1);

        \draw[decoration={brace, mirror, raise=20pt},decorate]
  (3,-2) -- node[below=20pt] {$m$} (5,-2);

    \end{tikzpicture}
    
    &   \begin{tikzpicture}[square/.style={regular polygon,regular polygon sides=4},scale=1]
    \draw[opacity=0] plot coordinates {
                (0, -1)
            } ;
        \node at (1,0) [circle, draw] (a2) {};
        \path[-]
        (a2) edge[loop right] node {$1$} ();
    \end{tikzpicture}

  \\ \hline
  \end{tabular}}
\end{table}

\section{Decorated graphs and effective boundedness} \label{s5}

We start this section by defining a new graph which will allow us to find a bound for $Z$. We recall that $Z$ is the number of $E_i$s such that $E_i \cdot C=1$. 

\begin{definition}
An $E_i$ is \textbf{maximal} if $E_i\cdot C=1$, and if $E_i$ is not contained in any other $E_j$ with $E_j \cdot C=1$.  
\label{max}
\end{definition}

Recall that $l$ is the number of non-ADE T-singularities on $W$.

\begin{lemma}
There are at most $l$ maximal $E_i$.   
\end{lemma}

\begin{proof}
Two distinct $E_j$'s are either disjoint or one is contained in the other. By the classification in Theorem~\ref{classificationEC=1}, we also know that any maximal $E_i$ contains an ending $(-2)$-curve of some T-chain.
\end{proof}

We now define a particular decorated graph corresponding to the configuration of $\pi$- and $\sigma$-exceptional curves on $X$. 
\begin{definition} \label{decorated-graph}

First, to each T-chain on $X$ we assign a vertex. Connect two vertices by an edge if there exists a maximal $E_i$ such that the corresponding T-chains intersect the same $(-1)$-curve $F$. Label this edge with the number of $E_j\subset E_i$ satisfying $E_j\cdot C=1$ and
$E_j$ contains both $F$ and the ending $(-2)$-curve of exactly one T-chain. Next, if the given $E_i$ uses all ending $(-2)$-curves in all of the T-chains intersecting $E_i$, then we add a loop to one of the vertices. Label this loop in such a way that the sum of weights on all edges and loops is equal to the number of $E_j\subset E_i$ such that $E_j\cdot C=1$. 
\end{definition}

Following this process, one can check that we obtain the graphs visualized in Tables~\ref{table1},~\ref{table2},~\ref{table3}.  So, if we add the numbers assigned to the edges of the whole decorated graph, then we obtain $Z$, i.e. the total number of $E_i$'s with $E_i \cdot C= 1$.

We will pay attention to each of the connected components of this decorated graph. 

\begin{proposition}
Let $G'$ be a connected component of the decorated graph constructed above. Then $G'$ has three options:
\begin{itemize}
    \item[(G1)] It is a tree, and so it is formed by $l'$ vertices connected by $l'-1$ edges, and all of its $E_i$'s are of types T.2.1 and T.2.2 (as in Table~\ref{table1}). 
    \item[(G2)] There is one cycle and no loops. Therefore we have $l'$ vertices and $l'$ edges, and all of its $E_i$'s are of types T.2.1 and T.2.2.
    \item[(G3)] There is exactly one loop, and there are no other cycles. Therefore there are $l'$ vertices and $l'$ edges, and all except one maximal $E_i$ is type T.2.1 or T.2.2. 
\end{itemize} 
\label{shape}
\end{proposition}

\begin{proof}
The key observations for this proof are: Every maximal $E_i$ uses some ending $(-2)$-curves, we can have $(-2)$-curves only at one end of a T-chain, and all the maximal $E_i$ are disjoint. Let us divide the maximal $E_i$s according to Theorem~\ref{classificationEC=1} into builder type for (T.2.1) and (T.2.2), and non-builder type for the rest.

If there are no loops or cycles, then we are in case (G1).

Assume that there is a cycle in $G'$. By the classification in Theorem~\ref{classificationEC=1}, one can check that all edges in the cycle correspond to an $E_i$ of builder type. Now take a vertex $V$ in the cycle, and consider a vertex $V'$ outside of the cycle and adjacent to $V$. Then $V'$ must correspond to a T-chain with ending $(-2)$-curves that are contained in some maximal $E_i$ corresponding to the edge $VV'$, as the ending $(-2)$-curves in the T-chain corresponding to $V$ have been already used by another maximal $E_i$. Therefore, by Theorem~\ref{classificationEC=1}, we have that $VV'$ corresponds to an $E_i$ of builder type. In this way, there is no way to have another cycle and/or a vertex with a loop.

Assume now that $G'$ has a vertex with a loop. Consider the decorated subgraph $G_0$ (associated to that loop) given by a corresponding $E_i$ of non-builder type. Let $V'$ be a vertex in $G'$ not in $G_0$ and adjacent to $V \in G_0$. Then, $V'V$ corresponds to an $E_i$ of builder type as there are no other $(-2)$-curves in the T-chains of $G_0$ that we can use to construct another maximal $E_i$. Therefore, as in the previous paragraph, there is no way to have a cycle and/or another vertex with a loop. \end{proof}

\begin{remark}  For the combinatorially-inclined readers, Proposition~\ref{shape} is a restatement of the following fact about directed graphs: Every vertex in a connected digraph has outdegree at most one if and only if the directed graph is a pseudotree, i.e. it has at most one cycle/loop. In our case we can turn each connected component of our decorated graph into a digraph by replacing edges with arcs as follows: having an arc beginning at a vertex $V$ means that the $(-1)$-curve in the given $E_i$ is attached to an ending $(-2)$-curve of the $T$-chain corresponding to $V$. Having a loop at a vertex means that all of the $(-2)$-curves in the corresponding T-chains are used in that $E_i$. Because a given T-chain can only have $(-2)$-curves at one end, the outdegree of each vertex in this digraph is indeed at most one.
\label{shape for combinatorists}
\end{remark}


As the next step in our argument (Propositions~\ref{bgraph0} through Remark~\ref{bgraph8}) we study the number of $E_j$ with $E_j \cdot C=1$ and $E_j \subset E_i$, for every possible $G_{E_i}$ classified in Theorem~\ref{classificationEC=1} and displayed in Tables~\ref{table1},~\ref{table2},~\ref{table3}. The goal is to bound that number with respect to their T-chains.

\begin{remark} This is an observation we will use frequently when $K_S$ is nef. Let $\Gamma$ be a $\P^1$ in $X$. By the adjunction formula, we have $K_X \cdot \Gamma=-2-\Gamma^2$. Let $\Delta$ be a $(-1)$-curve in $X$, and assume $\Delta \cdot \Gamma=m$. Then after blowing-down $\Delta$, we obtain that the intersection of the canonical class with the image of $\Gamma$ is $-2-\Gamma^2-m$. Therefore, if $K_S$ is nef, then $\Gamma^2 \leq -(\sum_i m_i)-2$, where the $m_i$s are the multiplicities corresponding to the various blow-downs.
\label{blowdowns}
\end{remark}

\begin{remark}
For any T-chain $[x_1,x_2,\ldots,x_r]$, we have (see e.g. \cite{RU17}) $$r-d+2=\sum_{i=1}^r(x_i-2).$$ We will use this a lot in the next propositions, together with the iterative description of T-chains given in Proposition~\ref{T-chain}.
\label{formula}
\end{remark}


\begin{proposition}
Assume that $K_S$ is nef. If we have a maximal $E_i$ of type T.2.1, 
then we have $r_1-d_1\geq m$ and $r_2-d_2\geq m$, where $r_k,d_k$ are the values of the corresponding T-chains in $E_i$.
\label{bgraph0}
\end{proposition}

\begin{proof}
The T-chain which has curves in $E_i$ is one of three possibilities: $$[2,\ldots,2,4+m]$$
$$[2,\ldots,2,3,2,\ldots,2,3+m]$$
$$[2,\ldots,2,x_1,\ldots,x_h,2+m],$$ with $x_1>2$. In the last case, we note that there must be $j\neq 1$ such that  $x_j>3$. So in any case, applying Remark~\ref{formula} gives $r_1-d_1\ge m$, with equality only in the first two cases. 

Since $K_S$ is nef, the curve in the other T-chain that intersects the $(-1)$-curve must have self-intersection at most $-(m+2)$. 
Hence, by Remark~\ref{formula}, we have $r_2-d_2+2\geq m$. 

If $r_2-d_2+2=m$, then this T-chain is $[2,\ldots,2,m+2]$. We compute the discrepancies of the curves intersecting the $(-1)$-curve $F$ using Proposition~\ref{T-chain}, and obtain that they are  $\delta_1=-1+\frac{m+1}{m+2}$ and $\delta_2=-1+\frac{1}{m}$. Since $\delta_1+\delta_2>-1$, this implies that $F$ is negative for $K_W$, a contradiction.

If $r_2-d_2+2=m+1$, then the options for the T-chain are $[2,\ldots,2,m+3]$, $[m+2,2,\ldots,2,3,2,\ldots,2]$, or $[2,m+2,3]$ (here $m$ is forced to be $3$). The last case contradicts the maximality of $E_i$, and in fact the given $E_i$ would be contained in a maximal $E_j$ of type T.2.5. For the first two of these cases we again compute the discrepancies of the curves attached to $F$ and obtain a contradiction as the image of $F$ in $W$ is negative for $K_W$. 
So $r_2-d_2 \geq m$.
\end{proof}

\begin{proposition}
If we have a maximal $E_i$ of type T.2.2, 
then we have $r_1-d_1\geq n+m$ and $r_2-d_2\geq n+m$, where $r_k,d_k$ are the values in the T-chains that intersect curves in $E_i$. For at least one $i$ we have $r_i-d_i \geq m+n+1$.
\label{bgraph1}
\end{proposition}

\begin{proof}

Suppose to begin with that $m=1$. Let us assume for a contradiction that we have the following particular case  with $n$ arbitrary $$[x_1,\ldots,x_{r_1-1},n+2]-(1)-[2,\ldots,2,n+4].$$  Note that the discrepancy of the ending $(-2)$-curve on the right T-chain is $-1+\frac{n+1}{n+2}$. Observe that the first T-chain can be neither $[2,\ldots,2,n+2]$ nor $[2,\ldots,2,3,2,\ldots,2,n+2]$, since both would imply that $E_i$ contains more curves in this T-chain, contradicting that $m=1$. 
Hence this T-chain is of the form $$[x_1,\ldots,x_{r_1-1},n+2]=[\underbrace{2,\ldots,2}_n,A,\ldots,B,n+2].$$ 

Suppose that this T-chain is obtained from  $[A-1,\ldots, B]$ with $t_1=a$ and $t_j=b$ (corresponding to the last curve) 
as in Proposition~\ref{T-chain}.
Then 
the ending $(-n-2)$-curve has discrepancy 
$$-1+\frac{a+b}{(n+2)a+(n+1)b}.$$
But
$$\left(-1+\frac{a+b}{(n+2)a+(n+1)b}\right) +\left(-1+ \frac{n+1}{n+2}\right)> -1,$$
so the image of $(-1)$-curve is not $K_W$-ample, a contradiction.



In this way, if $m=1$, then the T-chain on the right must have been obtained from $[4]$ or $[3,2,\ldots , 2, 3]$ by first adding a $(-2)$-curve on the right, taking a number of additional steps, and then adding $n$ $(-2)$-curves on the left. In particular, we get $\sum (y_i-2) \ge 3+n$, so by Remark~\ref{formula}, we have $r_2-d_2 \geq n+1$.

Thus, for $m=1$, we have the T-chains $$[x_1,\ldots,x_{r_1-1},n+2]-(1)-[2,\ldots,2,y_1,\ldots, y_{r_2-1-n},n+2].$$ And since the curve corresponding to $x_{r_1-1}$ is not part of $E_i$, we must have $x_{r_1-1}>2$ as well; in particular, the T-chain on the left was also obtained by first adding a $(-2)$-curve on the right, taking a number of additional steps, and then adding $n$ $(-2)$-curves on the left.  
Thus, by Remark~\ref{formula}, we obtain $$r_1-d_1+2= n+\sum_{i=1}^{r_1-1} (x_i-2),$$ and this is bigger than or equal to $n+3$ as the center contributes at least an additional $3$ to the sum. Therefore $r_1-d_1 \geq n+1$. 

Note that if $r_i-d_i=n+1$ for both chains, then we have that the T-chains are both equal to $[2,\ldots,2,5,n+2]$. Hence for at least one $k$ we have $r_k-d_k \geq n+1+1=n+2$. 

Now we do induction on $j$ the number of times we contract, within $E_i$, a chain of curves of the form $[s+2]-(1)-[2,\ldots,2]$ (with $s$ twos). Note that $m=1$ corresponds to $j=1$. So if $j>1$, then we have $$[2,\ldots,2,x_1,\ldots,x_{r'_1},s+2]-(1)-[2,\ldots,2,y_1,\ldots, y_{r'_2},s+2],$$ with $x_{r'_1}=2$ and $y_1>2$. Let us consider the related chain of curves $$[x_1-1,\ldots,x_{r'_1}]-(1)-[y_1-1,\ldots, y_{r'_2}],$$
and observe that this chain is also an $E_i$ of type T.2.2, with $j$ lowered by $1$.


Then, by the inductive hypothesis, we have $r'_k-d'_k \geq n'+m'$ for $k=1,2$. But $r_k=r'_k+s+1$, $d'_k=d_k$, $m'=m-1$, and $n'=n-s$, and so $$r_k-d_k=r'_k+s+1-d'_k \geq m'+n'+s+1= m+n.$$

Applying that at the beginning of the induction we have $k$ such that $r_k-d_k \geq n+2$, we obtain by induction that one $r_k-d_k$ is at least $m+n+1$. 
\end{proof}

\begin{proposition}
Assume that $K_S$ is nef. If we have a maximal $E_i$ as in case T.2.3, 
then we have $r_1-d_1\geq m$, $r_2-d_2 \geq m$, where $r_1,d_1$ corresponds to the T-chain with $m$ $(-2)$-curves, and $r_2,d_2$ to the T-chain with $n$ $(-2)$-curves. 
\label{bgraph2}
\end{proposition}

\begin{proof}
By Remark~\ref{formula}, we have $r_1-d_1+2\geq m+2$. By the same formula, we have $r_2-d_2+2 \geq m+1$. If $r_2-d_2+2 =m+1$, then $n=1$ and the T-chain is $[2,m+2,3]$, and so $m=3$. But then $E_i$ is not maximal. Therefore $r_2-d_2+2\geq m+2$.

\end{proof}

\begin{proposition}
If we have a maximal $E_i$ of type T.2.4, 
then we have $r_1-d_1\geq 2$, where $r_1,d_1$ are the values corresponding to the T-chain which is not $[4]$. 
\label{bgraph3}
\end{proposition}

\begin{proof}
This is trivial using Remark~\ref{formula}.
\end{proof}

\begin{proposition}
If we have a maximal $E_i$ of type T.2.5, 
then we have $r_1-d_1 \geq 3$, where  $r_1,d_1$ are the values for the T-chain which is not $[2,5,3]$. 
\label{bgraph4}
\end{proposition}

\begin{proof}
This is trivial using Remark~\ref{formula}.
\end{proof}


\begin{proposition}
Assume that $K_S$ is nef. If we have a maximal $E_i$ of type T.3.1,
then we have $r_1-d_1\geq m$, $r_2-d_2=m+n-1$, $r_3-d_3\geq n$, where $r_1,d_1$ for the T-chain with $m$ $(-2)$-curves, $r_2,d_2$ for the T-chain with $m+n-1$ $(-2)$-curves, and $r_3,d_3$ for the T-chain with $n$ $(-2)$-curves.
\label{bgraph5}
\end{proposition}

\begin{proof}
By Remark~\ref{formula}, we have $r_1-d_1\geq m$,  
$r_2-d_2+2=m+n+1$ and $r_3-d_3 \geq n$.  
\end{proof}

\begin{remark}
The bounds in Proposition~\ref{bgraph5} are strong enough for our argument. However, a closer inspection of the T-chains shows that in fact we have $r_1-d_1\ge m+n$ and $r_3-d_3\ge m+n$. 
\end{remark}

\begin{proposition} 
Assume $K_S$ is nef. If we have a maximal $E_i$ of type T.3.2, 
then we have $r_1-d_1\geq m$, $r_2-d_2= m+n-1$ and $r_3-d_3\geq m+n-1$, where  the pair  $r_1,d_1$ is for the T-chain with $m$ $(-2)$-curves, the pair $r_2,d_2$ for the T-chain that is completely contained in $E_i$, and the pair $r_3,d_3$ for the other T-chain.
\label{bgraph6}
\end{proposition}

\begin{proof}
This is Remark~\ref{formula} as usual, and we use that $K_S$ is nef. That $r_2-d_2=m+n-1$ is evident from the diagram. For $r_3-d_3$ we obtain in principle $r_3-d_3\geq m+n-2$. But equality implies $m=2$, $r_3=n+2$, $d_3=1$. But in this case, the discrepancy for the corresponding $(-1)$-curve does not work for the ampleness of $K_W$. Therefore $r_3-d_3\geq m+n-1$. 

\end{proof}

\begin{proposition}
Assume that $K_S$ is nef. If we have a maximal $E_i$ of type C.1, 
then we have $r_1-d_1\geq 2m$, where $r_1,d_1$ are the values for the T-chain.
\label{bgraph7}
\end{proposition}

\begin{proof}
This is precisely \cite[Lemma 2.15]{RU17}, and uses that $K_S$ is nef. 

\end{proof}

\begin{remark}
For maximal $E_i$ of type C.2 
there is no properly contained $E_j \subset E_i$ with $E_j \cdot C=1$ and $r-d\geq m+1$ by Remark~\ref{formula}.
\label{bgraph8}
\end{remark}

We now analyse some special properties of graphs, which will come in handy to join all the information from the bounds in each possible maximal $E_i$.

\begin{lemma} Let $G=\langle V, E\rangle$ be a tree with vertex set $\{V_1,\ldots, V_p\}$. For each edge $V_iV_j\in E$, let $G'$ be the subgraph of $G$ obtained by removing the edge $V_iV_j$, and define $w_{ij}$ to be the number of vertices in the connected component of $G'$ that contains the vertex $j$. Then for each $i$ we have $$\sum_{j=1}^pw_{ij}=p-1,$$ and for all $V_iV_j\in E$, we have $w_{ij}+w_{ji}=p$.
\label{aij on tree}
\end{lemma}

\begin{proof} Observe that for fixed $i$ and edge $V_iV_j$, the sum $\sum_{j=1}^pw_{ij}$ is the total number of vertices in the subgraph of $G$ obtained by removing $V_i$ and its incident edges. But this just counts all vertices in $V$ except $V_i$, and thus $\sum_{j=1}^pw_{ij}=p-1$.

Next, observe that $w_{ij}+w_{ji}$ simply counts the total number of vertices in the subgraph of $G$ which consists of two connected components and is obtained by removing the edge $V_iV_j$, and so $w_{ij}+w_{ji}=p$.
\end{proof}

\begin{remark}If we have a tree, we can visualize the values $w_{ij}$ described in Lemma~\ref{aij on tree} as weights on the corresponding symmetric digraph. 
In this context, Lemma~\ref{aij on tree} states that the sum of the weights on each pair of arcs connecting two vertices adds up to $p$, and the outdegree of each vertex is $p-1$. 
\end{remark}

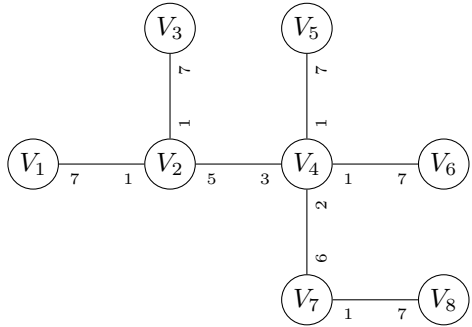
\begin{figure}[H]
\centering
    \begin{tikzpicture}[square/.style={regular polygon,regular polygon sides=7},scale=0.9]

        \node at (0,0) [circle, draw,inner sep=2pt] (a2) {$V_2$};
        \node at (2,0) [circle, draw,inner sep=2pt] (a4) {$V_4$};
        \node at (2,2) [circle, draw,inner sep=2pt] (a5) {$V_5$};       
        \node at (4,0) [circle, draw,inner sep=2pt] (a6) {$V_6$};     
        \node at (-2,0) [circle, draw,inner sep=2pt] (a1) {$V_1$};    
        \node at (0,2) [circle, draw,inner sep=2pt] (a3) {$V_3$};

        \node at (2,-2) [circle, draw,inner sep=2pt] (a7) {$V_7$};
        \node at (4,-2) [circle, draw,inner sep=2pt] (a8) {$V_8$};

        \path

        (a1) edge node[below] {\tiny{$7$ \hspace{0.4cm}  $1$}} (a2)
        (a2) edge node[below] {\tiny{$5$ \hspace{0.4cm}  $3$}} (a4)
        (a2) edge node[below,rotate=90] {\tiny{$1$ \hspace{0.4cm} $7$}} (a3)
        (a4) edge node[below,rotate=90] {\tiny{$1$ \hspace{0.4cm} $7$}} (a5)
        (a4) edge node[below] {\tiny{$1$ \hspace{0.4cm}  $7$}} (a6)
        (a4) edge node[below, rotate=90] {\tiny{$6$ \hspace{0.4cm}  $2$}} (a7)
        (a7) edge node[below] {\tiny{$1$ \hspace{0.4cm}  $7$}} (a8);

    \end{tikzpicture}
    \caption{Visualization of Lemma~\ref{aij on tree} in an example. Observe that the values at each vertex sum to $p-1=7$, and the sum of values on each edge is $p=8$.} \label{example.combi}
\end{figure}

\begin{corollary} Let $G=\langle V, E\rangle$ be a finite tree, where $V=\{V_1, \ldots, V_p\}$ is the vertex set. Assume that for each $V_i\in V$ and edge $V_iV_j\in E$, we have assigned nonnegative real numbers $a_i$ and $z_{ij}$ satisfying $a_i\ge z_{ij}=z_{ji}$. Then we have the following inequality
$$(p-1)\sum_{j=1}^p a_j\ge p \sum_{V_iV_j\in E}z_{ij}.$$
\label{aibij on tree}
\end{corollary}

\begin{proof}
By Lemma~\ref{aij on tree}, we can assign positive integers $w_{ij}$ satisfying $\sum_{j=1}^p w_{ij}=p-1$ and $w_{ij}+w_{ji}=p$. So fix an edge $\{i,j\}\in E$. Then by assumption we have
$a_j\ge z_{ij}$.
Multiplying both sides of this inequality by $w_{ji}$, we obtain the inequality
$$w_{ji}a_j\ge w_{ji}b_{ij}$$

Summing over $i$ and using that $\sum_{i=1}^p w_{ji}=p-1$, we obtain
$$
a_j\cdot (p-1)\ge \sum_{i=1}^p w_{ji}z_{ij}.$$
Now taking a sum over $j$, we have
$$ (p-1)\sum_{j=1}^p a_j \ge \sum_{j=1}^p\sum_{i=1}^p w_{ji}z_{ij}
= \sum_{V_iV_j\in E}(w_{ij}+w_{ji})z_{ij}= p \sum_{V_iV_j\in E}z_{ij},$$
where here we have used that $w_{ii}=0$, and $z_{ij}=z_{ji}$.
\end{proof}

\begin{lemma} Let $G=\langle V, E \rangle$ be a finite tree. Then there is a labeling $\{V_1,\ldots, V_p\}$ of the vertices of $G$ so that there is a bijection $\{V_2,\ldots V_p\}\to E$ such that each vertex maps to an edge containing it, i.e. for each $i$ we have $V_i\mapsto V_i V_j$ for some $j$. (In this labeling, vertex $V_p$ corresponds to a leaf of the tree.)
\label{tree bijection}
\end{lemma}

\begin{proof} We induct on the number of vertices. The statement is trivial if $|V|=2$. Now suppose the lemma holds for $p-1$ vertices and let $G=\langle V, E \rangle$ be a finite tree with vertex set $\{V_1,\ldots, V_p\}$. Relabel the graph so that $V_p$ corresponds to a leaf, i.e. a vertex with only one edge. Removing $V_p$ and its incident edge $e$, we obtain a subgraph $G'=\langle \{V_1,\ldots, V_{p-1}\}, E\setminus\{e\} \rangle$ which by the inductive hypothesis gives us a bijection $\{V_2,\ldots V_{p-1}\}\to E\setminus\{e\}$ that extends to $e$ and satisfies the desired property.
\end{proof}

\begin{corollary} Let $G=\langle V, E\rangle$ be a finite tree with vertex set $\{V_1,$ $\ldots,$ $ V_p\}$. Assume we have assigned to each vertex $V_i\in V$ and edge $V_iV_j\in E$ the nonnegative real numbers $a_i$ and $z_{ij}=z_{ji}$ such that $a_i\ge z_{ij}$. Then up to relabeling the vertices, we have
$$(p-1)a_1+(2p-1)\sum_{i=2}^pa_i\ge 2p\sum_{V_iV_j\in E}z_{ij}.$$
\label{key graph inequality}
\end{corollary}

\begin{proof}
By Lemma~\ref{tree bijection}, we can label the vertices in such a way that there is a bijection from $\{V_2,\ldots, V_p\}$ to $E$ which sends each $i\in V\setminus\{V_1\}$ to an edge containing $V_i$. So for each $V_j\in V$, let $V_iV_j\in E$ be the associated edge. For each fixed $j$, then, we have one inequality $a_j\ge z_{ij}$. Summing over $j$ gives
$$\sum_{j=2}^pa_j\ge \sum_{j=2}^p z_{ij}.$$
Observing that the right-hand side is in fact a sum over all edges in $E$, and multiplying both sides of the inequality by $p$ we have
$$p\sum_{j=2}^pa_j\ge p\sum_{V_iV_j\in E} z_{ij}.$$
Adding this to the inequality obtained in Corollary~\ref{aibij on tree}, we obtain the desired result.
\end{proof}

\begin{remark} In what follows, we apply Corollaries~\ref{aibij on tree} and ~\ref{key graph inequality} to the decorated graph $G$ defined near the beginning of this section. In particular, given a vertex $V$ in $G$ and incident edge $VV'$, we will take $a_V$ to be $r_V-d_V$ and $a_{V'}:=r_{V'}-d_{V'}$. Our goal is then to determine a value $z_{VV'}=z_{V'V}$ for the corresponding edge that satisfies the hypothesis of these corollaries. We will see that the assigned labels described in the beginning of this section give us a starting point.

\end{remark}

By Proposition~\ref{shape}, we have three possibilities for a connected component $G'$ of the decorated graph. We refer to them as G1, G2, and G3 as in that proposition. Let us fix notation. The graph $G'$ has $l'$ vertices. Let $R'-D'$ be the sum of the $r_i-d_i$ corresponding to each of the vertices in $G'$, and let $Z'$ be the sum of the numbers assigned to the edges of $G'$. We now obtain bounds for each of the three possibilities of $G'$. 

We will always assume that $K_S$ is nef from now on. 

\begin{proposition}
Let $G'$ be of type G1. Then $(l'-1)(R'-D') \geq l' Z'$. 
\label{G1bound}
\end{proposition}

\begin{proof}
In this case, we only have T.2.1 and T.2.2 maximal $E_i$'s. By Propositions~\ref{bgraph0} and~\ref{bgraph1}, for each vertex $V$ in $G'$, we have $r_V-d_V \geq z$ where $r_V-d_V$ corresponds to the vertex $V$ and $z$ is the value at the edge corresponding to $V$ via the bijection in Lemma~\ref{tree bijection}. Then we directly apply Corollary~\ref{aibij on tree} with $a_V:=r_V-d_V$ and $z_{V V'}:=z$.     \end{proof}

\begin{proposition}
Let $G'$ be of type G2. Then $$(2l'-1)(R'-D') \geq 2l'Z'-l'.$$ Not all maximal $E_i$'s in the cycle are of type (T.2.2). 
\label{G2bound}
\end{proposition}

\begin{proof}
In this case, $G'$ has only T.2.1 and T.2.2 maximal $E_i$'s, it has a cycle $\mathscr{C}$ of them, and it must have at least two $E_i$'s. First, we observe that not all the maximal $E_i$'s in $\mathscr{C}$ are of type T.2.2. Indeed, if they are all of type T.2.2, then we look at the discrepancies associated to the connecting $(-1)$-curves in $\mathscr{C}$. Say that the ending curves of the T-chains in the cycle have discrepancies $\delta_1,\ldots, \delta_{2s}$, where $\delta_1,\delta_2$ correspond to a T-chain in the cycle, $\delta_3,\delta_4$ to the next T-chain, and so on, until $\delta_{2s-1},\delta_{2s}$. Say that the $(-1)$-curves are $A_1,A_2, \ldots, A_s$, and that $A_1$ connects to curves with discrepancies $\delta_{2s},\delta_1$, $A_2$ to $\delta_2,\delta_3$, $A_3$ to $\delta_4, \delta_5$, and so on. Since $K_W$ is ample, we need to have $\delta_{2i-2}+\delta_{2i-1}<-1$ for all $A_i$. But, by Proposition~\ref{T-chain} part (iv), we know that $\delta_1+\delta_2=-1$, $\delta_3+\delta_4=-1$, ..., $\delta_{2s}+\delta_1=-1$. Then $$-s=(\delta_{2s}+\delta_1)+(\delta_2+\delta_3)+\ldots+(\delta_{2s-2}+\delta_{2s-1}) <-s,$$ a contradiction. 

Let us consider the cycle $\mathscr{C}$ with consecutive vertices $V_1,\ldots,V_s$. Let $t_i$ be the number of ending $(-2)$-curves corresponding to the T-chain of the vertex $V_i$. We note that all of these $(-2)$-curves are contracted by the maximal $E_i$s in this cycle. Assume that for $V_i - V_{i+1}$ the $t_i$ $(-2)$-curves are contained in $E_i$ (maybe there are more curves contracted by $E_i$). 

Note that by the above argument, there exists an edge in $\mathscr{C}$ corresponding to maximal $E_i$ of type T.2.1, where the $(-1)$-curve in $E_i$ does not intersect two ending curves of the corresponding T-chains. By relabeling, let $V_1-V_2$ be the subgraph of $\mathscr{C}$ corresponding to this $E_i$ and suppose that the T-chain corresponding to $V_1$ contributes its $t_1$ ending $(-2)$-curves to its $E_i$.


Let $t$ be the minimum of the $t_i$s. We claim there exists a subgraph $V_{j} - V_{j+1}$ of type T.2.1 with $t_j=t$. Suppose not, then any subgraph $V_{j} - V_{j+1}$ with $t_j=t$, must be of type T.2.2 and not of type T.2.1. In the proof of Proposition~\ref{bgraph1}, we showed that the two T-chains in an $E_i$ of type T.2.2 have the same number of ending $(-2)$-curves. So if we let $V_j-V_{j+1}$ be a subgraph of type T.2.2 with $t$ ending $(-2)$-curves used in $V_j$, then the T-chain corresponding to $V_{j+1}$ also has $t$ ending $(-2)$-curves, which it must contribute to the $E_i$ corresponding to subgraph $V_{j+1}-V_j$ in order to continue around the cycle. In this way, we obtain $t=t_j=t_{j+1}$, and inductively, that there are no type T.2.1 $E_i$s in the cycle $\mathscr{C}$, a contradiction.

Now let $F$ be the $(-1)$-curve in the $E_i$ corresponding to the subgraph $V_1-V_2$. Then $F$ intersects a curve $A$ in the T-chain corresponding to $V_2$ with $A^2\le -(t_1+2)$; in fact either $A^2\le -(t_1+3)$ or there exists another non-ending curve in the center of the T-chain with self-intersection at most $-3$. Since the T-chain corresponding to $V_2$ has $t_2$ ending $(-2)$-curves, the other ending curve $A'\neq A$ satisfies $A'^2\le -(t_2+2)$. Thus, $r_{V_2}-d_{V_2}+2\ge t_1+t_2+1\ge 2t+1$. So $r_{V_2}-d_{V_2}\ge t_1+t_2-1\ge 2t-1$.



Let us erase from $G'$ the edge $V_{j} - V_{j+1}$. As we have only one cycle, the new decorated graph $G' \setminus V_{j} - V_{j+1}$ is a connected tree with $l'$ vertices. By Corollary~\ref{key graph inequality} we have 
$$(l'-1)a_{V_2} + (2l'-1) \sum_{V\neq V_2} a_{V} \geq 2l' \sum_{\textrm{edges } VV' \neq V_{j}V_{j+1}} z_{VV'},$$ 
where $a_V:= r_V -d_V$ and $z_{VV'}$ is the number assigned to the edges. Therefore we obtain $$ (2l'-1)(R'-D') \geq 2l'(Z'-t)+l'(r_{V_2}-d_{V_2})=2l'Z'-l' + l'(r_{V_2}-d_{V_2}-2t+1).$$ 
Since $r_{V_2}-d_{V_2}-2t+1\geq 0$, we obtain the result.
\end{proof}

\begin{proposition}
Let $G'$ be of type G3. We recall that there is one maximal $E_i$ which gives a vertex in G3 with a loop. Then we have the following inequalities depending on the type of this unique $E_i$:

\begin{itemize}
    \item (T.2.3), (T.3.1), (T.3.2), (C.2): $(l'-1)(R'-D') \geq l' (Z'-1)$. 
    \item (T.2.4): $(l'-1)(R'-D') \geq l' (Z'-2)$. 
    \item (T.2.5): $(l'-1)(R'-D') \geq l' (Z'-3)$. 
    \item (C.1): $(2l'-1)(R'-D') \geq 2l'Z'$.
\end{itemize}

\label{G3bound}
\end{proposition}

\begin{proof}
Let us first consider (T.2.3), (T.3.1), (T.3.2), (C.2). For each of these cases, we have that for every vertex $V$ in $G'$ we have that $a_V:=r_V-d_V \geq z_{VV'}$, where $z_{VV'}$ is the number assigned to the edge $VV'$. This is true because of Propositions~\ref{bgraph0},~\ref{bgraph1},~\ref{bgraph2},~\ref{bgraph5},~\ref{bgraph6}. We now can apply Corollary~\ref{aibij on tree} to $G'$ minus the only one loop, to obtain $(l'-1)(R'-D') \geq l' (Z'-1)$, because the loop is decorated with the number $1$.

In the other cases to analyze we have either problems with $r_V-d_V \geq z_{VV'}$ for the vertex with a self-intersecting loop, and/or the number assigned to the loop is bigger than $1$.   

Let us consider (T.2.4). For the corresponding maximal $E_i$ with a loop at $V$, we have that $r_V-d_V=1-1=0$. Therefore we assign $z_{VV'}=0$,  and then apply Corollary~\ref{aibij on tree} to $G'$ minus the loop. We obtain $(l'-1)(R'-D') \geq l' (Z'-2)$. (Note that on the right-hand side, we have $Z'-2$ because we've removed the loop and labeled the edge with $0$ instead of $1$.) The case (T.2.5) is similar; we have $r_V-d_V=3-1=2$ at the special vertex $V$. So we assign $z_{VV'}=2$ (instead of $3$), and then apply Corollary~\ref{aibij on tree} to $G'$ minus the loop. We obtain $(l'-1)(R'-D') \geq l' (Z'-3)$.   

Let us consider (C.1). Here we first erase the loop at $V$ (which is labeled with $m$) from $G'$. Then we apply Corollary~\ref{key graph inequality} with $V_1=V$, and obtain $$ (2l'-1)(R'-D') \geq 2l'(Z'-m)+l'(r_{V_1}-d_{V_1})=2l'Z' + l'(r_{V_1}-d_{V_1}-2m),$$ but we know by Proposition~\ref{bgraph7} that $r_{V_1}-d_{V_1}\geq 2m$.
\end{proof}

Let us now exhibit a global bound for $Z$. We decompose the decorated graph $G$ into connected components $G'_i$ with $l'_i$ vertices. We define $l(G'_i)=2l'_i$ if $G'_i$ is of type $G2$ or $G3$ with a $(C.1)$; $l(G'_i)=l'_i$ otherwise. Moreover, we define:

\[   
f(G'_i) = 
     \begin{cases}
        \frac{1}{2} &\quad \text{if $G2$,}\\
       1 &\quad \text{if $G3$ with loop (T.2.3), (T.3.1), (T.3.2), (C.2), }\\
       2 &\quad\text{if $G3$ with loop (T.2.4),} \\
       3 &\quad\text{if $G3$ with loop (T.2.5),} \\
       0 &\quad\text{otherwise.} \\ 
     \end{cases}
\]

As before, let $R'_i-D'_i$ be the sum of the $r_V-d_V$ in $G'_i$, and $Z'_i$ be the sum of the values as edges. 

\begin{proposition}
    We have that $$Z \leq R-D - \sum_i \frac{1}{l(G'_i)} (R'_i-D'_i) + \sum_i f(G'_i).$$
    \label{bound}
\end{proposition}

\begin{proof}
    This is just adding the $Z'_i$ from Propositions~\ref{G1bound},~\ref{G2bound}, and~\ref{G3bound}.
\end{proof}

\begin{theorem}
If $K_S$ is nef, then $$\sum_i \frac{1}{l(G'_i)} (R'_i-D'_i) \leq 2(K_W^2-K_S^2)+ \sum_i f(G'_i) - K_S \cdot \pi(C).$$
\label{maintheorem}
\end{theorem}

\begin{proof}
We have $R-D \leq 2(K_W^2-K_S^2) + Z-\lambda$ by Proposition~\ref{r-d}. Then it follows directly from Proposition~\ref{bound}.
\end{proof}

\begin{corollary}
    If $K_S$ is nef, and the decorated graph $G$ has only connected components of type G1, G2, and G3 with a loop from a (C.1), then $$R-D \leq 4L(K_W^2-K_S^2) - 2L K_S \cdot \pi(C),$$ where $L\leq l$ is the maximum number of vertices in a $G'_i$.
\end{corollary}

\begin{corollary}
If $K_S$ is nef, and the decorated graph $G$ has only connected components of type G1, then $$R-D \leq 2L(K_W^2-K_S^2) - L K_S \cdot \pi(C),$$ where $L\leq l$ is the maximum number of vertices in a $G'_i$.
\end{corollary}

\section{Applications of the bounds}\label{s6}

In this final section, we state some applications of Theorem~\ref{maintheorem} and its corollaries. We begin with a combinatorial, asymptotically optimal example. We then finish the paper by applying the proof of Theorem~\ref{maintheorem} to obtain stronger bounds for two T-singularities, followed by a list of optimal examples in this case.

\begin{example}
For arbitrary $n,l$ positive numbers, let us have $l$ T-chains $$C_i=[2,\ldots,2,3,n+3,X_i, 2,\ldots,2,3,n+2],$$ where each $2,\ldots,2$ represents a chain of $n$ twos, $X_1=5$ and for $i>1$ $X_i=3+i,2,\ldots,2,3$ which has $i-2$ twos. 

Let us also have a $(-1)$-curve intersecting transversely the $(-n-3)$-curve in $C_1$ and the $(-2)$-end-curve in $C_1$, and for $i<l$, let us have a $(-1)$-curve intersecting transversely the $(-n-2)$-end-curve in $C_i$ and the $(-2)$-end-curve in $C_{i+1}$. Let us assume that all curves contracted by $\pi$ are already in Figure~\ref{Optimal Case}. 

After we contract all possible curves by $\pi$, the image of $C$ consists of the chain $$[n+2,3,2,\ldots,2,3,\ldots,3,4,n+3,2],$$ together with an extra $(-2)$-curve intersecting the last $(-2)$-curve at two points. The $2,\ldots,2$ has $n$ twos, and $3,\ldots,3$ has $l-1$ threes.

\begin{figure}[htbp]
\scalebox{0.8}{
\centering
    \begin{tikzpicture}[square/.style={regular polygon,regular polygon sides=4},scale=0.8]

        \node at (0,0) [square, draw, label=above:-n-2] (a1) {};
        \node at (1,0) [square,draw,label=above:-3] (a2) {};
        \node at (2,0) [square,draw,label=above:-2] (a3) {};
        \node at (3,0) [draw=none] (a4) {\ldots};
        \node at (4,0) [square,draw,label=above:-2] (a5) {};
        \node at (5,0) [circle, draw,fill=black, label=above:-5] (a6) {};
        \node at (6,0) [circle, draw,fill=black, label=above:-n-3] (a7) {};
        \node at (7,0) [circle, draw, fill=black, label=above:-3] (a8) {};
        \node at (8,0) [square, draw, label=above:-2] (a9) {};
        \node at (9,0) [draw=none] (a10) {\ldots};
        \node at (10,0) [square, draw, label=above:-2] (a11) {};

        \node at (8,-1) [circle, draw, label=above:-1] (a12) {};

        \draw (a1) -- (a2) -- (a3) -- (a4) -- (a5) -- (a6) -- (a7) -- (a8) -- (a9) --(a10) -- (a11);
        \draw (a11) -- (a12) -- (a7);

        \draw[decoration={brace,raise=20pt},decorate]
  (2,0) -- node[above=20pt] {$n$} (4,0);
        \draw[decoration={brace,raise=20pt},decorate]
  (8,0) -- node[above=20pt] {$n$} (10,0);

        \node at (11,2) [square, draw, label=above:-n-2] (b1) {};
        \node at (10,2) [square,draw,label=above:-3] (b2) {};
        \node at (9,2) [square,draw,label=above:-2] (b3) {};
        \node at (8,2) [draw=none] (b4) {\ldots};
        \node at (7,2) [square,draw,label=above:-2] (b5) {};
        \node at (6,2) [square,draw, label=above:-3] (b6) {};
        \node at (5,2) [circle, draw,fill=black, label=above:-5] (b7) {};
        \node at (4,2) [square,draw, label=above:-n-3] (b8) {};
        \node at (3,2) [square, draw, label=above:-3] (b9) {};
        \node at (2,2) [square, draw, label=above:-2] (b10) {};
        \node at (1,2) [draw=none] (b11) {\ldots};
        \node at (0,2) [square, draw, label=above:-2] (b12) {};
        \node at (1,1) [circle, draw, label=above:-1] (b13) {};

        \draw (b1) -- (b2) -- (b3) -- (b4) -- (b5) -- (b6) -- (b7) -- (b8) -- (b9) -- (b10) -- (b11) -- (b12) --(b13)-- (a1);

        \draw[decoration={brace,raise=20pt},decorate]
  (0,2) -- node[above=20pt] {$n$} (2,2);
        \draw[decoration={brace,raise=20pt},decorate]
  (7,2) -- node[above=20pt] {$n$} (9,2);

        \node at (14,6) [circle,draw,fill=black, label=above:-n-2] (c1) {};
        \node at (13,6) [circle,draw,fill=black,label=above:-3] (c2) {};
        \node at (12,6) [circle,draw,fill=black,label=above:-2] (c3) {};
        \node at (11,6) [draw=none] (c4) {\ldots};
        \node at (10,6) [circle,draw,fill=black,label=above:-2] (c5) {};
        
        \node at (9,6) [circle,draw,fill=black,label=above:-3] (c6) {};
        \node at (8,6) [circle,draw,fill=black,label=above:-2] (c7) {};
        \node at (7,6) [draw=none] (c8) {\ldots};
        \node at (6,6) [circle,draw,fill=black,label=above:-2] (c9) {};
        \node at (5,6) [circle, draw,fill=black, label=above:-3-l] (c10) {};
        \node at (4,6) [square,draw, label=above:-n-3] (c11) {};
        \node at (3,6) [square, draw, label=above:-3] (c12) {};
        \node at (2,6) [square, draw, label=above:-2] (c13) {};
        \node at (1,6) [draw=none] (c14) {\ldots};
        \node at (0,6) [square, draw, label=above:-2] (c15) {};
        
        \node at (1,5) [circle, draw, label=above:-1] (c16) {};

        \node at (10,3) [circle, draw, label=above:-1] (B) {};
        \node at (5,4) [draw=none] (A) {\ldots};

        \draw (c1) -- (c2) -- (c3) -- (c4) -- (c5) -- (c6) -- (c7) -- (c8) -- (c9) -- (c10) -- (c11) -- (c12) -- (c13)-- (c14) -- (c15) -- (c16) -- (A) -- (B) --(b1);

        \draw[decoration={brace,raise=20pt},decorate]
  (0,6) -- node[above=20pt] {$n$} (2,6);
        \draw[decoration={brace,raise=20pt},decorate]
  (6,6) -- node[above=20pt] {$l-2$} (8,6);
        \draw[decoration={brace,raise=20pt},decorate]
  (10,6) -- node[above=20pt] {$n$} (12,6);

    \end{tikzpicture}
}
    \caption{Asymptotically optimal but combinatorial example. } \label{Optimal Case}
\end{figure}
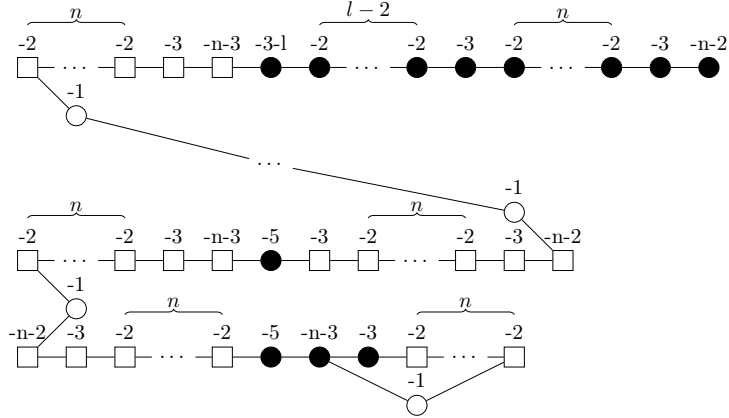

We have that $n_i=6n+5+8in+8i+2n^2+2in^2$, $a_i=3+2in+5i$, $r_i=4+i+2n$, and $d_i=1$ for each $C_i$ for every $i=1,\ldots,l$. One can check via Proposition~\ref{T-chain} that the images of the $(-1)$-curves in $W$ are positive with respect to $K_W$. So this is a combinatorial example whose decorated graph has only one connected component, and it is of type G3, shown in Figure~\ref{Optimal directed graph}.

\begin{figure}
\centering
    \begin{tikzpicture}[square/.style={regular polygon,regular polygon sides=4},scale=.6]

        \node at (0,0) [circle, draw] (a1) {};
        \node at (3,0) [circle, draw] (a2) {};
        \node at (6,0) [circle, draw] (a3) {};
        \node at (9,0) [draw=none] (a4)  {\ldots};
        \node at (13,0) [circle, draw] (a5) {};
        
        \path[-]
        (a1) edge node[below] {$2n+4$} (a2);
        \path[-]
        (a2) edge node[below] {$2n+5$} (a3);
      \path[-]
        (a3) edge node[below] {$2n+6$} (a4);
        \path[-]
        (a4) edge node[below] {$2n+l+2$} (a5);
        \path[-]
        (a1) edge[loop right] node {$n$} ();
\draw[decoration={brace,mirror,raise=10pt},decorate]
  (0,-.5) -- node[below=10pt] {$l$ vertices} (13,-.5);
    \end{tikzpicture}

\caption{The decorated graph of type G3 corresponding to the example displayed in Figure~\ref{Optimal Case}.}
\label{Optimal directed graph}
\end{figure}
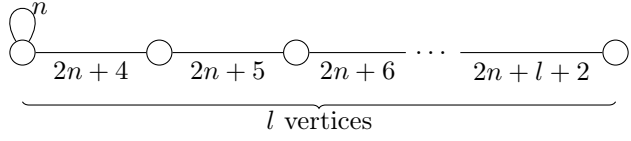

We have $R-D=2ln+\frac{l^2}{2}+\frac{7l}{2}$.
In the decorated graph we have a loop whose associated number is $n$, and the edges between $C_i$ and $C_{i+1}$ has the number $2n+3+i$ on it. Thus $$Z=n+(2n+3)(l-1)+\frac{(l-1)(l)}{2}=n(2l-1)+\frac{l^2}{2}+\frac{5l}{2}-3.$$ 

By Lemma~\ref{int} we have that $K_S \cdot \pi (C)=R-D+2l-Z-2l=n+l+3$ and $$K_W^2-K_S^2=R-D-m+l=R-D-Z=n+l+3.$$ Therefore $R-D=2(K_W^2-K_S^2)+Z-K_S \cdot \pi (C)$, and so equality holds in Proposition~\ref{r-d}.

In addition, we have $$\lim_{n \to \infty}\frac{R-D}{4l(K_W^2-K_S^2)-2l K_S\cdot\pi(C)}= \lim_{n \to \infty} \frac{7l+l^2+4nl}{4l(n+l+3)}= 1,$$ which implies that Theorem~\ref{maintheorem} is asymptotically optimal in a combinatorial sense.
\end{example}

An optimal bound for an arbitrary $l$ would be cumbersome, because of the potential connections between T-chains with maximal $E_i$'s. They affect immediately our simple vertex inequalities $a_V \geq z_{VV'}$, and so optimal bounds are lost in that process. But when one fixes the number of T-chains, one can use the work developed in this whole paper to obtain optimal bounds. The difficulty grows as we increase $l$. We will now show how to work out the case $l=2$. (The case $l=1$ was fully worked out in \cite{RU17}.)   

\begin{theorem}
Assume that $K_S$ is nef and $l=2$. Let us denote the vertices by $V_i$ with the corresponding $r_i,d_i$ for $i=1,2$. Then there are at most two maximal $E_i$ such that $E_i\cdot C=1$. We classify according to the number of such $E_i$s.
\begin{itemize}
    \item[(1)] There is no such $E_i$. Then $$ R-D \leq 2(K_W^2-K_S^2) - K_S \cdot \pi(C).$$ Equality holds if and only if for all $E_i$ we have $E_i \cdot C=2$.
    \item[(2)] There is exactly one such $E_i$. Case-by-case bounds are displayed in Table~\ref{oneEi}.
    \item[(3)] There are exactly two such $E_i$s. Case-by-case bounds are displayed in Table~\ref{twoEi}.
\end{itemize}
\label{optimal l=2}
\end{theorem}

\begin{table}
\begin{tabular}{|l|l|}
\hline
Type of $E_i$ & Bound on $R-D$\\
\hline
(C.1)\footnotemark &  $\displaystyle \frac{r_1-d_1}{2} + r_2-d_2 \leq 2(K_W^2-K_S^2) - K_S \cdot \pi(C) $ \\ \hline
(C.2) & $R-D \leq 2(K_W^2-K_S^2) - K_S \cdot \pi(C) +1.$ \\ \hline
(T.2.1) & $R-D \leq 4(K_W^2-K_S^2) - 2 K_S \cdot \pi(C)$. \\ \hline
(T.2.2) & $R-D \leq 4(K_W^2-K_S^2) - 2 K_S \cdot \pi(C)-1$. \\ \hline
(T.2.3) & $R-D \leq 4(K_W^2-K_S^2) - 2 K_S \cdot \pi(C)$.\\ \hline
(T.2.4)\footnotemark & $R-D \leq 2(K_W^2-K_S^2) - K_S \cdot \pi(C)+2$ \\ \hline
(T.2.5)\footnotemark & $R-D \leq 4(K_W^2-K_S^2) - 2 K_S \cdot \pi(C)+5$ \\ \hline
\end{tabular}
\caption{Bounds on $R-D$ for two T-singularities, assuming exactly one maximal $E_i$. See Theorem~\ref{optimal l=2}.}\label{oneEi}
\end{table}
\addtocounter{footnote}{-2}
\footnotetext{The T-chain with numbers $r_1,d_1$ is part of (C.1).}
\addtocounter{footnote}{1}
\footnotetext{We have $r_2-d_2=0$.}
\addtocounter{footnote}{1}
\footnotetext{We have $r_2-d_2=2$.}

\begin{table}
\begin{tabular}{|l|l|}
\hline
Types of $E_i$ & Bound on $R-D$\\
\hline
(C.1)+(C.1) &  $\displaystyle R-D \leq 4(K_W^2-K_S^2) - 2 K_S \cdot \pi(C).$ \\ \hline
(C.2)+(C.2) &  $\displaystyle R-D \leq 2(K_W^2-K_S^2) - K_S \cdot \pi(C) +2.$\\ \hline
(C.1)+(C.2)\footnotemark &  $\displaystyle\frac{r_1-d_1}{2} + r_2-d_2 \leq 2(K_W^2-K_S^2) - K_S \cdot \pi(C) +1.$ \\ \hline
(T.2.1)+(T.2.1)\footnotemark & $R-D\leq 6(K_W^2-K_S^2) - 3 K_S \cdot \pi(C)+2$  \\ \hline
(T.2.1)+(T.2.1)\footnotemark & $R-D \leq 4(K_W^2-K_S^2) - 2 K_S \cdot \pi(C)$   \\ \hline
(T.2.1)+(T.2.2) & $R-D \leq 8(K_W^2-K_S^2) - 4 K_S \cdot \pi(C)-1$ \\ \hline
(T.2.1)+(C.1)\footnotemark & $R-D \leq 4(K_W^2-K_S^2) - 2 K_S \cdot \pi(C) +1$ \\ \hline
(T.2.1)+(C.1)\footnotemark & $R-D \leq 6(K_W^2-K_S^2) - 3 K_S \cdot \pi(C) +1$ \\ \hline
(T.2.2)+(C.1) & $R-D \leq 8(K_W^2-K_S^2) - 4 K_S \cdot \pi(C)$\\ \hline
(T.2.1)+(C.2) & $R-D \leq 4(K_W^2-K_S^2) - 2 K_S \cdot \pi(C) +1$\\ \hline
(T.2.2)+(C.2) & $R-D \leq 4(K_W^2-K_S^2) - 2 K_S \cdot \pi(C)+2$ \\
\hline
\end{tabular}
\caption{Bounds on $R-D$ for two T-singularities, assuming two maximal $E_i$. See Theorem~\ref{optimal l=2}.}\label{twoEi}
\end{table}
\addtocounter{footnote}{-4}
\footnotetext{Here (C.1) corresponds to $r_1,d_1$, and (C.2) with $r_2,d_2$.  }
\addtocounter{footnote}{1}
\footnotetext{This is the bound in the case that a $(-1)$-curve intersects two ending curves.}
\addtocounter{footnote}{1}
\footnotetext{This is the bound if neither $(-1)$-curve intersects two ending curves.}
\addtocounter{footnote}{1}
\footnotetext{This is the bound if the $(-1)$-curve connected to both T-chains intersects only one ending curve.}
\addtocounter{footnote}{1}
\footnotetext{This is the bound if the $(-1)$-curve connected to both T-chains intersects two ending curves.}

\begin{proof}
We have exactly three distinct possibilities for the maximal $E_i$: (1), (2), and (3). The first possibility (1) is the inequality in Proposition~\ref{r-d} with $Z=0$, and from its proof one sees that equality holds if and only if for all $E_i$ we have $E_i \cdot C=2$.

For (2) and (3) we will first work out an upper bound for $Z$, to then apply it using the inequality $R-D \leq 2(K_W^2-K_S^2) - K_S \cdot \pi(C) + Z$ in Proposition~\ref{r-d}. This upper bound is particularly worked out for $l=2$, and it ends up being better than the general bound we have for arbitrary $l$. We analyze everything case by case. 

For the cases that have only (C.1) and (C.2), we independently use the inequalities we have from \cite{RU17} for (C.1) and (C.2). In the case of (C.1), we have $Z=m$ and $2m\leq r-d$. For (C.2) we just use $Z=1$. The remaining cases in (2) and (3) are more complicated. We work them out below. 

\bigskip 
\textbf{(T.2.1).} This is direct from Theorem~\ref{maintheorem}.

\bigskip
\textbf{(T.2.2).} Directly from Proposition~\ref{bgraph1}, we have that $2Z \leq R-D -1$, because $Z=n+m$.  

\bigskip
\textbf{(T.2.3).} We have from Proposition~\ref{bgraph2} that $R-D \geq 2m$, since $r_i-d_i\geq m$ for each $i$. We want to show that $R-D\geq 2m+2$. We recall that $Z=m+1$ in this case. We have two sub-cases to analyze:

\textbf{(Case 1):} Assume that $r_1-d_1=m$. In this case, the first T-chain is $[m+4,2,\ldots,2]$. As always, the second T-chain is $[2,\ldots, 2,m+2,x,\ldots,y]$ with $n$ 2's. 
If this second T-chain has length $n+2$, then it must be $[2,\ldots,2,3,n+3]$ or $[2,\ldots,2,5,n+2]$. In the first case, computing discrepancies of the curves attached to the $(-1)$-curve, we obtain a contradiction with $K_W$ ample. 

In the second case, we have $m=3$ and so $R-D=m+n+1=n+4$. Suppose for a contradiction that  $R-D=n+4<2m+2=8$. Then $n=1,2$ or $3$. If $n=1$ then $E_i$ is not maximal. If $n=2$ or $n=3$, the ending curve in the first T-chain becomes $K_S$-negative, again a contradiction. 

Therefore the second T-chain has length at least $n+3$, and we obtain $r_2-d_2+2\geq m+x-2+n$. As $K_S$ is nef, we have $x-1-n\geq 1$, and so $r_2-d_2 \geq m+2n-2$. If $n\geq2$, then $r_2-d_2 \geq m+2$, and we are done. Otherwise $n=1$, and one checks that $r_2-d_2 \geq m+2$ anyways. 

\textbf{(Case 2):} Assume that $r_1-d_1=m+1$. We have that the first T-chain is either $[m+2,3,2,\ldots,2,4,2,\ldots,2]$ or $[m+2,5,2,\ldots,2]$. We want to show that $r_2-d_2 \geq m+1$. We have that the second T-chain has the form $[2,\ldots, 2,m+2,x,\ldots,y]$ and $x \geq n+2$.

If the first T-chain is $[m+2,3,2,\ldots,2,4,2,\ldots,2]$, then $n=1$ and after contracting the $(-1)$-curve in $X$ followed by the ending $(-2)$-curves, the $(-4)$-curve becomes a $(-1)$-curve intersecting the now $-(x-2)$-curve with multiplicity $2$ 
Using nefness of $K_S$, we obtain that $x\ge 6$, and thus $r_2-d_2\ge m+2>m+1$, as desired. 

If the first T-chain is $[m+2,5,2,\ldots,2]$, then $n=1$ or $2$. If $n=2$, then after contracting the $(-1)$-curve and the ending $(-2)$-curves, the $(-5)$-curve becomes a $(-1)$-curve intersecting the now $(-x+3)$-curve  with multiplicity $3$, so since $x-3-9\ge1$, we get $x\ge 13$, and so $r_2-d_2\ge m+11$. If $n=1$, then one checks directly that $r_2-d_2 \geq m+1$.

Therefore $r_1-d_1 \geq m+2$, and so we have what we wanted. 

\bigskip
\textbf{(T.2.4).} This is direct from Proposition~\ref{r-d}, because $Z=2$. 

\bigskip
\textbf{(T.2.5).} This is direct from Proposition~\ref{r-d}, because $Z=5$. 

\bigskip
\textbf{(T.2.1)+(T.2.1).} Let us denote by $r_i,d_i,m_i$ the corresponding parameters for each (T.2.1). It is enough to show that $2Z \leq R-D$, ie. $R-D\geq 2(m_1+m_2)$.

Note that by nefness of $K_W$, at most one $(-1)$-curve in these $E_i$ intersects ending curves of both T-chains. So first, we assume that each $(-1)$-curve in these $E_i$s intersects a curve which is not an ending curve. Then following the proof of Theorem~\ref{bgraph0} and using the formula in Remark~\ref{formula}, we obtain that $r_i-d_i \geq m_1+m_2-2$.

If $r_i-d_i \geq m_1+m_2$ for $i=1,2$, then we are done. 

If $r_1-d_1=m_1+m_2-2$ (the case $i=2$ is analog), then the T-chain is $$[2,\ldots,2,m_2+2,2,\ldots,2,m_1+2].$$ 
Then from the other T-chain we have a connecting $(-1)$-curve on the left, giving the chain of curves $(1)-[2,\ldots,2,m_2+2,2,\ldots,2,m_1+2]$ 
But then after contracting the chain $[2,\ldots,2]-(1)$ connected to the $(-m_2-2)$-curve, we get $(1)-[2,\ldots,2]-(1)$, and this contradicts nefness of $K_S$.

\bigskip 

If $r_1-d_1=m_1+m_2-1$, then this T-chain is one of 
\begin{itemize}
    \item $[2,\ldots,2,m_2+2,2,\ldots,2,m_1+3]$, 
    \item $[2,\ldots,2,3,m_2+2,2,\ldots,2,m_1+2]$, or
    \item $[2,\ldots,2,m_2+3,\underbrace{2,\ldots,2}_{m_2-2},m_1+2]$.
\end{itemize} 

Indeed, the first of these options contradicts the fact that $K_S$ is nef, following the exact same reasoning as above. We will show that the second and third imply $r_2-d_2\ge m_1+m_2+1$, giving us the desired inequality.

For the second, observe that in order for this to be a T-chain, it must be of the form
$$[2,\ldots, 2, 3, m_2+2, m_1+2]$$ 
and so in particular, $m_2=3$ and this T-chain is $[2, \ldots, 2, 3, 5, m_1+2]$.  Then the second T-chain is
$$[m_2+2=5,\ldots,x\geq m_1+3,\ldots,y\ge 3,2,2,2].$$

If $x$ and $y$ correspond to the same curve, then by nefness of $K_S$ we obtain that $x\geq m_1+6$, so we have $r_2-d_2\ge m_1+4+m_2-2=m_1+m_2+2$. 

If $x$ and $y$ correspond to different curves, then we immediately have $r_2-d_2 \geq m_2 +m_1+1 + 1 -2=m_1+m_2$. If $r_2-d_2=m_1+m_2$, then $x=m_1+3$ and $y=3$. But then the curve corresponding to $y$ eventually becomes a $\P^1$ with $0$ self-intersection, but this contradicts the assumption $K_S$ nef.

For the third option, the second T-chain has the form $$[m_2+2,\ldots,x\geq m_1+2,\ldots,y,2,\ldots,2],$$
where the curve corresponding to $x$ is connected via a $(-1)$-curve to the first T-chain. If $x$ and $y$ correspond to the same curve, then as in the previous case, we obtain $x\ge m_1+6$, leading us to the desired inequality. Otherwise, using that $K_S$ is nef, we obtain that $y \geq m_2+5$. Indeed the $-(m_2+3)$-curve in the first T-chain will eventually become a $(-1)$-curve. Therefore, as $K_S$ is nef, we must have $y-1-1-(m_2-2)\geq 1$. Hence we obtain the inequality $r_2-d_2 \geq m_1+m_2+1 $.



Now suppose that one of the  $(-1)$-curves is connected to two end curves in the T-chains. Let $V_2$ be the T-chain that has $(-1)$-curves connected to both ends. By Proposition~\ref{bgraph0}, we have that $r_2-d_2\geq m_1$ and $r_2-d_2\geq m_2$. By the formula in Remark~\ref{formula}, we have that $r_1-d_1+2\geq m_1+m_2$. Like in the previous case we can discard the possibility of $r_1-d_1=m_1+m_2-2$.




Using that $2(r_2-d_2)\geq m_1+m_2$, we obtain that $2(R-D)\geq 3(m_1+m_2)-2=3Z-2$. Therefore we have the desired inequality as a consequence of Proposition~\ref{r-d}. 

\bigskip
\textbf{(T.2.1)+(T.2.2).} We observe that the (T.2.2) has $m,n>0$. 

Let $V_1$ be the vertex with only black dots when one considers the (T.2.1) maximal $E_i$. Then the ending $(-2)$-curves in the corresponding T-chain are part of the (T.2.2) maximal $E_i$ which contains the ending $(-2)$-curves, and this T-chain is  connected to the other T-chain via a black dot. A discrepancy computation shows that this connection is not an ending curve. Via the same proof as Proposition~\ref{bgraph1}, and using that we have an additional condition on some black dot in this T-chain, one can show that $r_1-d_1 \geq m+n+1$, where $m,n$ are the parameters associated to the (T.2.2). 

Therefore, from $r_1-d_1 \geq m+n+1$ and $r_2-d_2 \geq m+n$, we obtain $r_1-d_1 + 3(r_2-d_2) \geq 4(m+n)+1=4(Z-p)+1,$ where $p$ is the parameter for the (T.2.1). Therefore $3(R-D) \geq 4Z+1 + 2(r_1-d_1-2p)$, but we will show that $r_1-d_1\geq 2p$.

Let $m'$ be the number of ending $(-2)$-curves of $V_1$.
We have that $p\leq m'$ by using that these are T-chains and that $K_S$ is nef. 
We now observe that $r_1-d_1+2 \geq p + m'\ge 2p$ directly from our situation.  Therefore $r_1-d_1 \geq 2p-2$. By considering the curve attached to the $(-1)$-curve in the T.2.1 maximal $E_i$, we can use an analysis similar to that of case (T.2.1)+(T.2.1) above to show that $r_1-d_1=2p-2$ and $r_1-d_1=2p-1$ are impossible. 

\bigskip
\textbf{(T.2.1)+(C.1).} Let $V_1$ be the vertex with no loop. Let $m$ be the weight of the edge connecting $V_1$ and $V_2$ and let $n$ be the weight for the loop. 

Assume first that the $(-1)$-curve connecting $V_1$ and $V_2$ is connected to two ending curves.
Using formula~\ref{formula}, that $K_S$ is nef, and an analysis of the curves in the second T-chain that are adjacent to the maximal $E_i$s, we obtain that $r_2-d_2 +2 \geq (n+2-2) + (m+2-2) + (n+2-2) + (3-2)$, and so $r_2-d_2 \geq 2n+m-1$. On the other hand, we know that $r_1-d_1 \geq m$, and so $R-D \geq 2n+2m-2=2Z-1$. Then the inequality is a direct consequence of Proposition~\ref{r-d}.


Suppose now that the $(-1)$-curve in the T.2.1 is connected to two ending curves. By formula~\ref{formula} and $K_S$ being nef, we obtain that $r_2-d_2+2\geq (m+2-2)+(n+2-2)+(3-2)$, again by an analysis of the curves in the second T-chain that are adjacent to the maximal $E_i$s. We also have that $r_2-d_2\geq 2n$ by Proposition~\ref{bgraph7}. Therefore $2(R-D)\geq 2(m)+(m+n-1)+(2n) =3m+3n-1=3Z-1$. Therefore the inequality is a consequence of Proposition~\ref{r-d}.  

\bigskip
\textbf{(T.2.2)+(C.1).} This is a direct application of Theorem~\ref{maintheorem}. 

\bigskip
\textbf{(T.2.1)+(C.2).} Let $V_1$ be the vertex with no loop. The edge connecting $V_1$ and $V_2$ has weight $m$, i.e. assume the first T-chains has $m$ ending $(-2)$-curves. Let $n$ be the number of ending $(-2)$-curves in the second T-chain. We have $r_1-d_1 \geq m$ by Proposition~\ref{bgraph0}. Then one shows that $r_2-d_2 \geq m+2n-1$ via the same analysis as in previous cases.
Therefore, as $Z=m+1$, we obtain $2Z \leq R-D +1$, and so by Proposition~\ref{r-d} we get the claimed inequality.

\bigskip
\textbf{(T.2.2)+(C.2).} This is a direct application of Theorem~\ref{maintheorem}. 
    
\end{proof}

Almost all the inequalities in Theorem~\ref{optimal l=2} are sharp. Below we give examples for some of them. 

\begin{example}
It is easy to find optimal examples for (1) in Theorem~\ref{optimal l=2}. See e.g. \cite[Figure 1 left]{RU22} (which also has no obstructions to deformations), and 
\cite[Figure 8]{CFPRR23}.
\end{example}

\begin{example}
The following is an optimal example for (2) in Theorem~\ref{optimal l=2} part (C.1). Start with an elliptic K3 surface that has one $I_2$ fiber and $22$ $I_1$ fibers, and sections. This can be done via base change on a general  rational elliptic surface with sections. Consider the configuration given by one section, one $I_2$ and one $I_1$. Blow-up at the two nodes of the $I_2$, and three times over the node of the $I_1$. We obtain two chains:  $[4,2,6,2,2]$ and $[4]$. Their contraction $W$ has ample $K_W$. We have $\frac{r_1-d_1}{2}=2$, $r_2-d_2=0$, $K_S^2=0$, $K_S \cdot \pi(C)=0$, and $K_W^2=1$. Note that by \cite[Proposition 2.8]{RU22}, we have that $H^2(W,T_W)=\C$.
\end{example}

\begin{example}
This is an optimal example for (2) in Theorem~\ref{optimal l=2} part (C.2).  Consider a K3 surface with an elliptic fibration with sections and singular fibers $I_4+20I_1$. This is easy to construct via base change over a rational elliptic fibration with an $I_2$ and ten $I_1$ singular fibers. Assume it has two sections intersecting opposite curves $A,C$ in the $I_4$ fiber. Consider also another component $B$ of the $I_4$. The initial configuration is then $A,B,C$, one $I_1$, and two sections. We blow-up the node of the $I_1$ twice, the intersection of $A$ with $B$, and the intersection of the other section with the $I_1$. Then we obtain two T-chains $[2,6,2,3]$ and $[3,2,3]$. The $[2,6,2,3]$ is part of a (C.2). Contract them both to get $W$ with $K_W$ ample and $K_W^2=1$. By \cite[Proposition 2.8]{RU22} we have no obstructions to deform $W$.  
\end{example}







\begin{example}
Next, we give an optimal example for (3) in Theorem~\ref{optimal l=2} part (T.2.4). Consider a rational elliptic fibration with sections, and singular fibers: $II$ (rational cuspidal curve) and ten $I_1$. We do the base change of degree two over two nonsingular fibers. This is an elliptic K3 surface. We consider a section, one $II$, and one $I_1$. We blow-up the node of the $I_1$, and blow-up four times over the cusp of $II$, so that we have   $[4,2,6,2,2]$ and $[4]$ in a (T.2.4). If $W$ is the contraction of these two T-chains, then $K_W$ is ample and $K_W^2=1$. 
\end{example}


\begin{example}
This is an optimal example for (3) in Theorem~\ref{optimal l=2} case (T.2.1)+(T.2.1). Let us consider a rational elliptic fibration with sections, and singular fibers: $I_2$ plus ten $I_1$. Consider two disjoint sections intersecting both components of the $I_2$ fiber. Take a base change of degree $3$ at two nonsingular fibers. Then we have an elliptic fibration with sections of self-intersection $(-3)$. We obtain an elliptic fibration with three $I_2$ and $30$ $I_1$ as singular fibers. Consider the pre-image of the two sections and one of the $I_2$. Blow-up at the two nodes of the $I_2$ twice. We obtain two T-chains $[2,5,3]$. Their contraction is $W$. It has $K_W$ ample and $K_W^2=2$. One checks that equality holds in this example. 
\end{example}
  
\begin{example}
This is an optimal example for (3) in Theorem~\ref{optimal l=2} part (T.2.1)+(T.2.2), similar to the previous example. Start with an elliptic rational surface with sections and 12 $I_1$ fibers. Take base change of degree $3$ over an $I_1$ and over a nonsingular fiber. Then we have an elliptic fibration with sections of self-intersection $(-3)$ and an $I_3$ fiber. We blow-up twice over two of their nodes, and consider one section. We obtain T-chains $[2,5,3]$ and $[2,3,4]$. Their contraction defines a $W$ with $K_W$ ample and $K_W^2=1$. One checks that equality holds. This example appears in \cite[Figure 7]{CFPRR23}. 
\end{example}







\end{document}